\documentclass[%
 aip,
cp,   
 amsmath,amssymb,
 reprint,%
]{revtex4-2}

\usepackage{graphicx}
\usepackage{dcolumn}
\usepackage{bm}

\usepackage[utf8]{inputenc}
\usepackage[T1]{fontenc}
\newcommand{\CR}{\mathbb{C}_{Re>0}}
\usepackage{amsmath,amssymb,latexsym,amsmath,amsthm,amscd, enumitem}
\allowdisplaybreaks
\usepackage{xcolor}
\usepackage[margin=1in]{geometry}
\usepackage{pstricks}
\usepackage{graphicx} 
\usepackage{tikz}
\usetikzlibrary{ patterns}
\usepackage{mathrsfs}
\usepackage[normalem]{ulem}
\theoremstyle{plain}
\newtheorem{thm}{Theorem}
\newtheorem{lem}[thm]{Lemma}
\newtheorem{prop}[thm]{Proposition}
\newtheorem{cor}[thm]{Corollary}

\theoremstyle{definition}
\newtheorem{dfn}[thm]{Definition}

\theoremstyle{remark}
\newtheorem{rmk}[thm]{Remark}

\newcommand{\C}{\mathbb{C}}
\def\vac{\mathbf{1}}
\DeclareMathOperator{\End}{End}

\DeclareMathOperator{\Res}{Res}
\begin{document}

\title{On Rationality Of $\mathbb{C}$-Graded Vertex Algebras And Applications To Weyl Vertex Algebras Under Conformal Flow}

\author{Katrina Barron} 
 \email[Corresponding author: ]{kbarron@nd.edu}

 \affiliation{
  Department of Mathematics, University of Notre Dame, Notre Dame, IN 46556, USA. 
}

\author{Karina Batistelli}%
\email{kbatistelli@unc.edu.ar}
\affiliation{Department of Mathematics, Universidad de Chile, Santiago, Chile.}
\author{Florencia Orosz Hunziker}
\email{florencia.orosz@du.edu}
\affiliation{Department of Mathematics, University of Denver, Denver, CO 80208, USA.%
}%
\author{Veronika Pedić Tomić}
\email{vpedic@math.hr}
\affiliation{Department of Mathematics, University of Zagreb, Bijenicka cesta 30, Zagreb, Croatia.}
\author{Gaywalee Yamskulna}
\email{gyamsku@ilstu.edu}
\affiliation{Department of Mathematics Illinois State University, Normal, IL 61790, USA.}
\date{\today} 

\begin{abstract}
Using the Zhu algebra for a certain category of $\mathbb{C}$-graded vertex algebras $V$, we prove that if $V$ is finitely  $\Omega$-generated and satisfies suitable grading conditions, then $V$ is rational, i.e. has semi-simple representation theory, with one dimensional level zero Zhu algebra. Here $\Omega$ denotes the vectors in $V$ that are annihilated by lowering the real part of the grading.  We apply our result to the family of rank one Weyl vertex algebras with conformal element $\omega_\mu$ parameterized by $\mu \in \mathbb{C}$, and prove that for certain non-integer values of $\mu$, these vertex algebras, which are non-integer graded, are rational, with one dimensional level zero Zhu algebra.  In addition, we generalize this result to appropriate $\mathbb{C}$-graded Weyl vertex algebras of arbitrary ranks.
\end{abstract}

\maketitle

\section{1. Introduction}

In this paper, we study various subcategories of the category of $\mathbb{C}$-graded vertex algebras, including those with a conformal element imposing various grading structures.  We illustrate the nature of these subcategories via the conformal flow for the family of $\mathbb{C}$-graded Weyl vertex algebras with conformal elements $\omega_\mu$ parameterized by $\mu \in \mathbb{C}$.  We prove two rationality results for certain $\mathbb{C}$-graded vertex algebras that admit a conformal structure with a ``nice" grading property.  We then apply these results to show that for $\mu \in \mathbb{C}$ in a certain simply closed region of the complex plane, the corresponding Weyl vertex algebras with conformal element $\omega_\mu$ are rational (in the sense that the representation theory is semisimple), and in fact admit only one simple ``admissible" module, where ``admissible" here means having a grading compatible with that of the vertex algebra.  These admissible modules are also the modules that are induced from the level zero Zhu algebra.

A large portion of the literature on vertex algebras and their representations both from a mathematical and physical standpoint has been devoted to the study of rational conformal vertex algebras that are of CFT-type  (see, for instance, Section 1.1 of \cite{HLZ} for a list of these types of vertex algebras and references therein). It is a natural question to ask whether there are other significant classes of conformal vertex algebras that are well-behaved from the representation-theoretic point of view, for instance are either rational (have semi-simple representation theory) for some category of modules, or are irrational (have indecomposable modules that are not simple) for some category of modules, but the category has other nice properties. This is one of the motivations behind the concept of $\mathbb{C}$-graded vertex algebras.
\vfill
{\footnotesize K. Batistelli is supported by Fondecyt project 3190144. F. Orosz Hunziker is supported by the National Science Foundation under grant No. DMS-2102786. V. Pedić Tomić is partially supported by the QuantiXLie Centre of Excellence, a project
cofinanced by the Croatian Government and European Union through the European Regional Development Fund - the Competitiveness and Cohesion Operational Programme
(KK.01.1.1.01.0004). G. Yamskulna is supported by the College of Arts and Sciences, Illinois State University.}
\clearpage
 Conformal flow consists of the deformation of the conformal vector $\omega$ associated to a vertex operator algebra $V$ to obtain a new conformal structure $\omega_{\mu}$ on $V$,  for $\mu \in \mathbb{C}$ a continuous parameter. All possible conformal structures associated to the Heisenberg vertex algebra (also known as the free bosonic vertex algebra) were formally classified in  \cite{MN}.  One of  these ``shifted'' conformal structures for the Heisenberg vertex algebra is used in the study of the triplet algebras \cite{AM2}, an important example of $C_2$-cofinite but irrational vertex algebras. When deforming the conformal vector, the grading restrictions associated to the $L(0)$-operator are often lost. Namely, the new conformal vector $\omega_{\mu}$ does not necessarily satisfy that its zero mode $L^{\mu}(0)$ acts semisimply on $V$ or that each graded component  of $V$ must be finite dimensional. The appropriate framework to study the new conformal vertex algebra $(V, \omega_{\mu})$ is the theory of $\mathbb{C}$-graded vertex algebras developed in \cite{DM, LM} as a continuation of the development of the theory of $\mathbb{Q}$-graded vertex algebras started in \cite{DLM2}. Motivated in part by this work on $\mathbb{C}$-graded vertex algebras and conformal flow, where in \cite{LM} the notion of  ``$\mathbb{C}$-graded vertex algebra" is more specifically called an ``$\Omega$-generated $\CR$-graded vertex algebra" in our work, we establish a refinement of the various concepts of $\mathbb{C}$-grading for a vertex algebra.

The Weyl vertex algebra, to which we apply our results, admits a conformal flow. The Weyl vertex algebra has its origins in physics as fields of Faddeev-Popov ghosts in the early formulations of conformal field theory where it is also known in the physics literature (sometimes with some specific fixed central charge and thus conformal element) as the bosonic ghost system or the $\beta \gamma$-system (cf. \cite{P1}, \cite{P2}, \cite{FMS}, and references therein).    The terminology bosonic ghost system for the Weyl vertex algebra refers to the fact that this vertex algebra comprises one of the four fundamental free field algebras, those being free bosons, free fermions, bosonic ghosts, and fermionic ghosts. Consequently, the Weyl vertex algebra has played a crucial role  in many aspects of conformal field theory and the study of the various mathematical structures that conformal field theory involves. Conformal flow and the relationship between conformal flow for bosonic ghosts (i.e., the Weyl vertex algebra) and conformal flow of the free boson vertex subalgebra of bosonic ghosts was studied by Feigin and Frenkel in \cite{FF3}, and the BRST cohomology was calculated for certain Fock space representations of bosonic ghosts with  $\mu = 2$ and central charge $c=26$ associated to a 26-dimensional Minkowski space. The Weyl vertex algebra was used in the study of free field  realizations of affine Lie algebras and  the chiral de Rham complex  (cf. \cite{FF, FF1, FF2, FBZ, MSV, W}), and  more recently, Weyl vertex algebras have been used to describe relations between conformal field theory, topological invariants, and number theory through the study of the (twined) $K3$ elliptic genus and its connections to  umbral and Conway moonshine \cite{ACH}. 

 As discussed above, both free bosons and bosonic ghosts admit multiple conformal structures. In this paper, we give a detailed analysis of the nature of the conformal structures of the Weyl vertex algebra under conformal flow, classify the ``admissible" modules for the Weyl vertex algebra for certain infinite families of conformal elements, and prove that the category of such admissible modules is semisimple for these conformal structures.  We denote the Weyl vertex algebra by $M$, and the Weyl vertex algebra with conformal element $\omega_\mu$, for the complex parameter $\mu \in \mathbb{C}$, by $(_\mu M, \omega_\mu)$ or just $_\mu M$. 

The Weyl vertex algebra with conformal element $\omega_\mu = \omega_0$, denoted by $_\mu M = \,  _0 M,$ gives a conformal vertex algebra with central charge $c=2$, and has been studied intensively. For this conformal structure, the Weyl vertex algebra gives  a distinguished example of an irrational $\mathbb{Z}$-graded conformal vertex algebra,  of current interest in the setting of logarithmic conformal field theory.  The term ``irrational" refers to the fact that the conformal vertex algebra does not have semisimple representation theory, and logarithmic conformal field theory involves the study of such vertex algebras and the category structure of various types of modules for these algebras (cf. \cite{CG, GRR, AM}). In particular, categories of modules for which the zero mode of the conformal element $L(0)$ does not act semisimply even though the modules have certain nice $L(0)$-grading properties, often referred to as ``admissible", are the categories of interest, and specifically those closed under the tensor product and with graded characters that have modular invariance properties.     

It was shown by Ridout and Wood in \cite{RW} that $_0 M$ is not $C_2$-cofinite and admits reducible yet indecomposable modules on which the Virasoro operator $L(0)$ acts non-semisimply. Moreover, in \cite{RW}  the authors identified a module category $\mathcal{F}$  that satisfies three necessary conditions arising from logarithmic conformal field theory for the category to have a nice tensor structure. They also determined the modular properties of characters in  that category and computed the Verlinde formulas.  Then in \cite{AP},  Adamović and  Pedić computed the dimension of the spaces of intertwining operators among simple modules in category $\mathcal{F}$  and gave a vertex-algebraic proof of the Verlinde type conjectures in \cite{RW}. 
Recently in \cite{AW}, Allen and Wood classified all indecomposable modules in $\mathcal{F}$, showed that it is rigid and determined the direct sum decompositions for all fusion products of its modules.

 In \cite{ALPY} certain (nonadmissible) weak modules for the Weyl algebra with conformal element $\omega_\frac{1}{2}$ and central charge $c=-1$ were studied in the context of Whittaker modules and modules for the fixed point subalgebra of $_\frac{1}{2} M$ under a certain automorphism. Here, it was shown that the family of Whittaker modules described in \cite{ALPY} is irreducible for these orbifold (fixed point) subalgebras of the Weyl algebra at $\mu = 1/2$, while in a recent paper \cite{AP2}, the opposite was proved for other orbifold subalgebras where these Whittaker modules were shown to be reducible.  

A natural question to ask then is what is the nature of the category of admissible modules for the Weyl vertex algebra with conformal element other than $\omega_0$ under conformal flow, and more generally what broader results concerning the modules for non-integer graded conformal vertex algebras hold? In particular, is the category of admissible modules semisimple or not?

In this paper, we answer these questions.  In particular, we study  the influence of the central charge, or equivalently the choice of conformal element, on the  representation theory of Weyl vertex algebras of arbitrary rank in the case when the vertex algebra is not integer graded. More generally we study non integer graded conformal vertex algebras. We begin our investigation by studying the (level zero) Zhu algebra $A(V)$ of a finitely  $\Omega$-generated $ \CR $-graded vertex algebra $V$, where $\Omega$ denotes the vectors in $V$ that are annihilated by lowering the real part of the grading.  In fact, we show that if $V$ is an $\Omega$-generated $ \CR$-graded vertex algebra   that is finitely generated (in the usual sense) such that the generators do not have integer weights and $V$ contains an $\mathbb{N}$-graded  vertex subalgebra,  then $V$ is rational in the sense that the representation theory for admissible modules is semisimple. As an application, we  prove that in particular, a rank one Weyl vertex algebra $M_c$ with  $c\in\mathbb{R}$ and  $-1<c<2$ is rational. Consequently, we prove that a rank $n$ Weyl vertex algebra which  is a tensor product of $n$ rank one Weyl vertex  algebras, each with $c\in\mathbb{R}$ and  $-1<c<2$, is rational. More generally, we show that in fact for certain complex values of the central  charge under conformal flow, these rationality result holds as well. 

This phenomenon of the change in the nature of the representation theory of the conformal Weyl vertex algebra for admissible modules (i.e., modules compatible with the grading arising from the conformal structure) under conformal flow is surprising in contrast to the lack of change of the representation theory under conformal flow for the free boson vertex operator algebra.  See Remark \ref{conformalflowrem} below.

This paper is organized as follows: In Section 2, we define various notions involving vertex algebras with gradings, and/or with conformal vectors, and their modules. In Section 3, we study the rank one Weyl vertex algebra and the various graded structures imposed by the family of conformal vectors $\omega_\mu$, for $\mu \in \mathbb{C}$, under conformal flow with respect to $\mu$.  This family of conformal vertex algebras provides good examples and motivations for the various notions of vertex algebra defined in Section 2.  

In Section 4, we recall the notion of the Zhu algebra of  an $\Omega$-generated $\CR$-graded vertex algebra as introduced in \cite{LM}, where such vertex algebras were called $\mathbb{C}$-graded vertex algebras.  We also present several results on the correspondence between modules for the Zhu algebra $V$ and a certain class of $V$-modules, i.e. $\CR$-graded modules.  

In Section 5, we present our main results and applications to the Weyl vertex algebras. First we prove a theorem on the rationality of $\Omega$-generated $\CR$-graded vertex operator algebras satisfying certain conditions; see Theorem \ref{rationalcondition}.  

Then in Subsection 5.1, motivated by the work of Zhu \cite{Z} and of Li \cite{Li5} we define a filtration on the Zhu algebra of an $\Omega$-generated $\CR$-graded vertex algebra and prove that under this filtration we obtain a graded commutative associative algebra $grA(V)$.  We show that there is an epimorphism from our $\Omega$-generated $\CR$-graded vertex algebra to this graded commutative associative algebra with kernel of the epimorphism containing a set $C(V)$ which, in this setting, is an analogue of the set $C_2(V)$  defining the $C_2$-cofinite condition for a vertex operator algebra.  In Subsection 5.2, we give our main results on the rationality of certain $\CR$-graded vertex operator algebras with generators having noninteger weights by using the epimorphism  from $V/C(V)$ to $grA(V)$; see Lemma \ref{A(V)} and Theorem \ref{Vrational}.

In Subsection 5.3, we apply Theorems \ref{mu-theorem}, \ref{rationalcondition} and \ref{A(V)} to the Weyl vertex algebras with conformal vectors $\omega_\mu$ for $\mu$ in a certain region determined in Section 4 that give these vertex algebras the structure of an $\Omega$-generated $\CR$-graded vertex operator algebra, and prove that these are rational with only one $\CR$-graded module.  We then apply this result to the rank $n$ Weyl vertex algebras with suitable conformal element.  We also prove that more generally, for $\mu \in \mathbb{C} \smallsetminus \{0,1\}$ with $0\leq Re(\mu) \leq 1$, then the Weyl vertex algebra $_\mu M$ admits a unique, up to isomorphism, irreducible $\CR$-graded module, namely $_\mu M$ itself.

In  Section 6, we summarize the results of this paper and also present a result giving the level one Zhu algebra for $_0 M$, i.e. the Weyl vertex algebra with central charge $c=2$. 

\section{2. $\mathbb{C}$-graded vertex algebras and
their modules}

\subsection{2.1. Vertex algebras and  $\Omega$-generated $\mathbb{C}$-graded vertex algebras}

 We recall the definitions of various type of vertex algebras, following for instance \cite{LL, HLZ} for basic notions, but then motivated as well by the work of Laber and Mason in \cite{LM} in the setting of $\mathbb{C}$-graded vertex algebras and related notions. However, it should be noted that we use different terminology for some of the structures in \cite{LM};  cf. Remarks \ref{RmkLM1} and \ref{rms}. 
\begin{dfn}\cite{LL}
A \textit{vertex algebra} $(V, Y,\vac)$ consists of a vector space $V$ together with a linear map 
\begin{equation*}
		\begin{aligned}
		Y \colon V &\rightarrow (\mathrm{End}\,V)[[x,x^{-1}]],\\
		v &\mapsto Y(v,x)=\sum_{n\in\mathbb{Z}} v_n\,x^{-n-1},
		\end{aligned}
		\end{equation*}
		and a distinguished vector, $\vac \in V$(the vacuum vector), satisfying the following axioms:
		\begin{enumerate}[label=\textup{(\roman*)}]
			\item
			The lower truncation condition: for $v_1, v_2\in V$, $Y(v_1,x)v_2$ has only finitely many
			terms with
			negative powers in $x$.
			
			\item
			The vacuum property: $Y(\mathbf{1},x)$ is the identity endomorphism $1_V$ of $V$.
			\item The creation property: for $v\in V$,  $Y(v,x)\vac \in V[[x]]$ and $\lim_{x\rightarrow 0}Y(v,x)\vac=v$.

			\item
			The Jacobi identity: for $w,v\in V$,
			\begin{multline*}
			x_0^{-1}\delta\left(\dfrac{x_1-x_2}{x_0}\right) Y(v_1,x_1)Y(v_2,x_2)- x_0^{-1}\delta\left(\dfrac{x_2-x_1}{-x_0}\right)Y(v_2,x_2)Y(v_1,x_1) \notag \\
			 = x_2^{-1}\delta\left(\dfrac{x_1-x_0}{x_2}\right)Y(Y(v_2,x_0)v_1,x_2).
			\end{multline*}
\end{enumerate}			
\end{dfn}

\begin{dfn}\label{C-graded-va-def} 
A vertex algebra equipped with a $\mathbb{C}$-grading $V=\bigoplus_{\lambda \in\mathbb{C}}V_{\lambda}$ is called a {\em $\mathbb{C}$-graded vertex algebra} if ${\bf 1}\in V_0$ and if for $v\in V_{\gamma}$ with $\gamma\in\mathbb{C}$ and for $n\in\mathbb{Z}$, $\lambda\in\mathbb{C}$,
\begin{align} \label{wc}
v_n V_{\lambda}\subset V_{\lambda+\gamma-n-1}.    
\end{align}

Moreover, a  homogeneous element in a $\mathbb{C}$-graded vertex algebra $V$ is said to have {\em weight} $\lambda$ if  $v\in V_{\lambda}$. We denote this by $|v|=\lambda$, and we define the operator $L\in\End(V)$ as the linear extension of the map 
\begin{eqnarray} \label{Ldef}
V_{\lambda}&\rightarrow&V_{\lambda} \nonumber \\
v&\mapsto&\lambda v=|v|v.
\end{eqnarray} 
\end{dfn}

\begin{rmk} \label{weight-gradingin-remark} ${}$
\begin{enumerate}
    \item Since we do not require the existence of a conformal element in a $\mathbb{C}$-graded vertex algebra, the map defined above is a natural tool to describe the weight of a  homogeneous element. 
\item In a $\mathbb{C}$-graded vertex algebra because of Definition \ref{wc} we have that for $v^1, v^2 \in V$

    \begin{align*}
        |v^1_nv^2|=|v^1|+|v^2|-n-1.
    \end{align*}
    More generally, for $v, v^1, \dots, v^k\in V$
    \begin{align*} 
      |v^k_{n_k}\cdots v^1_{n_1}v|= \left(\sum_{j=1}^k|v^j|-n_j-1\right)+|v|.
    \end{align*}
\end{enumerate}
\end{rmk}

\begin{rmk} \label{RmkLM1}
In \cite{LM} the notion of $\mathbb{C}$-graded vertex algebras has more conditions than what we require above in Definition \ref{C-graded-va-def}.  In our terminology, the Laber-Mason notion of a $\mathbb{C}$-graded vertex algebra is  an $\Omega$-generated $\mathbb{C}$-graded vertex algebra, as defined below in Definition \ref{Cgrad}.  Many of our results in fact make fine distinctions between these two notions.  
\end{rmk}

\begin{rmk} \label{D} Recall from \cite{LL}, that for $V$ a vertex algebra, the endomorphism  $D:V\longrightarrow V$ defined as the linear map determined by $D(v) =v_{-2}\vac$ satisfies the $D$-derivative property: $Y(Dv,x) = \frac{d}{dx}Y(v,x)$.  Furthermore $D(\vac) = 0$, and $v = v_{-1} {\bf 1}$. It then follows that for a $\mathbb{C}$-graded vertex algebra, by Eqn.\ (\ref{wc}) and the $D$-derivative property, we have that if $v\in V_{\lambda}$,  then ${D}v = v_{-2} \vac \in V_{\lambda + 0 -(-2) -1} = V_{\lambda +1}$.
\end{rmk}

\begin{dfn} \label{omegaV} Let $V=\bigoplus_{\lambda \in\mathbb{C}}V_{\lambda}$ be a $\mathbb{C}$-graded vertex algebra.  We define 
\begin{eqnarray*}
\Omega(V) &=& \{ v \in V \; | \; \mbox{for  any $u\in V_{\gamma}$, $n \in \mathbb{Z}$, if $u_nv\neq 0$ then 
either $n=\gamma-1$ or $n< Re(\gamma)-1$} \} 
\end{eqnarray*}
where $Re(\gamma)$ denotes  the real part of $\gamma$. 
\end{dfn} 

\begin{rmk} ${}$

\begin{enumerate} 
    \item{The space $\Omega(V)$ consists of the vectors in $V$ that are zero if they are acted on by any mode of $V$ that lowers the real part of the weight.  This space is often called the ``vaccum space" or the ``space of lowest weight vectors".  However the vacuum vector ${\bf 1}$ is not necessarily in $\Omega(V)$.   For instance, assume that $V=\bigoplus_{\lambda\in\mathbb{C}} V_{\lambda}$ such that $V_{-10}\neq 0$. Let $a\in V_{-10}$. Notice that $a_{-1}{\bf 1}=a\neq 0$. Also, $-1\neq -10-1$ and $-1> Re(-10)-1$.  Hence,  in this case  ${\bf 1} \notin \Omega(V).$ We give an example of such a vertex algebra in Section 3, namely the Weyl vertex algebra $_\mu M$ with $\mu \in \mathbb{R}$ and $\mu<0$, for example $\mu = -1/2$ and thus $c=11$.} 
\item{In addition, the term ``lowest weight space" is   misleading since there can be vectors in $\Omega(V)$ that are not of lowest weight in the sense of having any kind of minimality property with respect to  their $\mathbb{C}$-grading in $V$;  instead, these are the vectors that can not be further lowered.  An example of such a $\mathbb{C}$-graded vertex algebra is, for instance, the universal Virasoro vertex operator algebra of central charge $c=\frac{1}{2}$, denoted $V_{Vir}(\frac{1}{2}, 0)$ (in the notation of \cite{LL}).  This $\mathbb{Z}$-graded vertex algebra is indecomposable but not irreducible, and it has a singular vector $v_{3,2}$ of
weight 6 that satisfies $v_{3,2}\in \Omega(V_{Vir}(\frac{1}{2},0))$. }
\end{enumerate}
\end{rmk} 

 Next we introduce the notion of  an $\Omega$-generated $\mathbb{C}$-graded vertex algebra  motivated by Laber and Mason \cite{LM},  where this notion is called a $\mathbb{C}$-graded vertex algebra.

\begin{dfn}\label{Cgrad} An { \it  $\Omega$-generated  $\mathbb{C}$-graded vertex algebra},  (or {\it a $\mathbb{C}$-graded vertex algebra generated by $\Omega$)} is a $\mathbb{C}$-graded vertex algebra $(V, Y, \mathbf{1})$ such that every element $v\in V$ is a finite sum of elements of the form  \[v^k_{n_k}v^{k-1}_{n_{k-1}} \cdots v^1_{n_1}u^0\]
for $k \in \mathbb{N}$, $n_1, \dots, n_k \in \mathbb{Z}$,  $v^1, \dots, v^k \in V$, and  $u^0 \in \Omega(V)$.

The notions of an $\Omega$-generated $\mathbb{R}$-graded,  $\Omega$-generated $\mathbb{Q}$-graded,  $\Omega$-generated $\mathbb{Z}$-graded and  $\Omega$-generated $\mathbb{N}$-graded vertex algebra are defined in the obvious way.
\end{dfn}
\begin{rmk}
We show in Section 3 that  the collection of $\Omega$-generated $\C$-graded vertex algebras form a proper subset of the set of $\C$-graded vertex algebras. Namely, in Section 3 we present a family of $\mathbb{C}$-graded Weyl vertex algebras which are not {\it $\Omega$-generated} $\mathbb{C}$-graded vertex algebras (see Theorem \ref{mu-theorem} III.)  
\end{rmk}

 We will also need the notions of a {\it strongly generated} and {\it finitely strongly generated} vertex algebra given as follows:

\begin{dfn} \label{sgen}
A {\it strongly generated vertex algebra} is a vertex algebra $(V,Y, {\bf 1})$ together with  a subset $S \subset V$ such that every element $v \in V$ is a finite sum of elements of the form 
\[v^k_{ -n_k}v^{k-1}_{ -n_{k-1}} \cdots v^1_{ -n_1}{\bf 1}\]
for $k \in \mathbb{N}$, $n_1, \dots, n_k \in \mathbb{Z}_{+}$,  and $v^1, \dots, v^k \in S$.  If $V$ is strongly generated by a finite set $S$, then we say that $V$ is {\it strongly finitely generated}.
\end{dfn}

\begin{rmk} \label{strogen}
Any $\Omega$-generated $\mathbb{C}$-graded vertex algebra $V$ is trivially a strongly generated  vertex algebra with $S=V$. If $V$ is also strongly finitely generated by a finite set of generators $S$ acting on $\Omega$ and $\Omega$ is also finite, then we call $V$ a {\it finitely} $\Omega$-generated $\mathbb{C}$-graded vertex algebra. All $\Omega$-generated $\mathbb{C}$-graded vertex algebras are strongly generated but the converse is not true, even if we have finitely many strong generators. In Theorem \ref{mu-theorem} (III) we give examples of finitely strongly generated $\mathbb{C}$-graded Weyl vertex algebras which are not $\Omega$-generated.
\end{rmk}

 For certain  $\Omega$-generated  $\mathbb{C}$-graded vertex algebras one can define a degree grading as follows in Definition \ref{degree} below, and we call such  $\Omega$-generated $\mathbb{C}$-graded vertex algebras  $\Omega$-generated $\CR$-graded vertex algebras. In Section 3 we give examples of  $\Omega$-generated $\mathbb{C}$-graded Weyl vertex algebras that admit a grading as defined below. 

\begin{dfn}\label{degree}
 An {\it  $\Omega$-generated $\CR$-graded vertex algebra} is an $\Omega$-generated $\mathbb{C}$-graded vertex algebra, such that the  following notion of degree is well defined:  For $V$ an  $\Omega$-generated $\mathbb{C}$-graded vertex algebra, we define the {\it degree} of an element of $V$  by setting the degree of elements in $\Omega(V)$ to be $0$, and extending by linearity the following formula:
 \[deg(v^k_{n_k}\cdots v^1_{n_1} u^0) = \sum_{j=1}^k (|v^j|-n_j-1)\]
where  $v^1, \dots, v^k \in V$, for $k \in \mathbb{N}$, $n_1, \dots, n_k \in \mathbb{Z}$, and  $u^0 \in \Omega(V)$.
\end{dfn}

\begin{rmk}
Note that this notion of degree is not necessarily well defined for every  $\Omega$-generated $\mathbb{C}$-graded vertex algebra. If $ v_n u^0 \in \Omega(V)$, then by definition of $\Omega (V)$, if $ v_n u^0 \neq 0$, then $deg( v_nu^0) =  |v| -n-1 = 0$ or $Re(deg( v_n u^0)) = Re(|v| - n - 1) = Re( |v|) - n - 1>0$. So by definition, this notion of degree, by setting all elements in $\Omega(V)$ to have degree zero, is precluding the possibility of elements in $\Omega(V)$ of the form $ v_n u^0$ such that $u^0 \in \Omega(V)$, and $ v_n u^0 \neq 0$ for some $n$ satisfying $Re( |v|) - n - 1 >0$. Thus it is the requirement of well-definedness of this definition that is imposing the degree grading given below. 
\end{rmk}

One can show (cf. \cite{LM}) that  it follows from  Definitions \ref{Cgrad} and \ref{degree} that:
\begin{lem}\label{grading-lemma} 
 Let $V$ be an  $\Omega$-generated  $\CR$-graded vertex algebra.  For 
$k\geq 1$, let $v^1$,\dots,$v^k \in V$ be homogeneous, $n_1 ,\dots,n_k \in \mathbb{Z}$, and $u^0 \in \Omega(V)$ such that \[v^k_{n_k}v^{k-1}_{n_{k-1}}\cdots v^1_{n_1}u^0\neq 0.\]
Then for any given $v^j \in V$ and $n_j \in \mathbb{Z}$, either 
 \[Re(\sum_{j=1}^k(| v^j|-n_j-1 ))>0 \qquad \text{ or } \qquad \sum_{j=1}^k(| v^j|-n_j-1)=0.\]
\end{lem}
\begin{proof} See Appendix A for a detailed proof of this fact. 
\end{proof}

   \begin{rmk} \label{rmkVadmmod}Note that if $V$ is an   $\Omega$-generated $\CR$-graded vertex algebra $V$ and we define   $V(\lambda)$ to be the space of all $v\in V$ with $deg(v)=\lambda$  then, we have the following decomposition
\begin{align} \label{Vismod}
V=V(0)\bigoplus_{\substack{\lambda \in \mathbb{C}\\ Re(\lambda)>0}} V(\lambda) .   
\end{align}
 This motivates our use of the term $\CR$-graded vertex algebras to denote this particular family of $\mathbb{C}$-graded vertex algebras.
\end{rmk}
\begin{prop}
 Let $V$ be  an $\Omega$-generated $\CR$-graded vertex algebra and let $deg$ be as in Definition \ref{degree}.  Then, the  homogeneous component $V(0)$ in (\ref{Vismod}) coincides with $\Omega(V)$.
\end{prop}
 \begin{proof}
Note that $\Omega(V)\subset V(0)$ follows from Definition \ref{degree}. Therefore, we need to show next that $ V(0)\subset \Omega(V) $. Namely, we need to prove that if $v\in V$ satisfies $deg(v)=0$ then $v$ must be a vector in $\Omega(V)$.
We first prove this fact for vectors of the form
\[v^k_{n_k} \cdots v^1_{n_1}u^0\] 
where $v^1, \dots, v^k \in V$,  for $k \in \mathbb{N},$
 $n_1, \dots, n_k \in \mathbb{Z}$,   and  $u^0 \in \Omega(V)$:

 Assume that $ v^k_{n_k} \cdots v^1_{n_1}u^0\neq 0$ and that  $deg( v^k_{n_k}\cdots v^1_{n_1}u^0)=\sum_{j=1}^k\left(|v^j|-n_j-1\right)=0$. We want to show that $ v^k_{n_k} \cdots v^1_{n_1}u^0 \in \Omega(V)$. Let $u\in V$ and $n\in \mathbb{Z}$ be such that $u_n  v^k_{n_k}\cdots v^1_{n_1}u^0\neq 0$. Then, using Lemma  \ref{grading-lemma} for  $u_n v^k_{n_k}\cdots v^1_{n_1}u^0$, we have that either $|u|-n-1+\sum_{j=1}^k\left( |v^j|-n_j-1\right)=0$ or $Re\left(|u|-n-1+\sum_{j=1}^k\left( |v^j|-n_j-1\right)\right)>0.$ Since by assumption $\sum_{j=1}^k\left( |v^j|-n_j-1\right)=0$ it follows that either $n=|u|-1$ or $n<Re( |u|-1)$. Therefore, $ v^k_{n_k}\cdots v^1_{n_1}u^0 \in \Omega(V)$ if $deg( v^k_{n_k}\cdots v^1_{n_1}u^0)=0$.  
 
 Now, let $v$ be any vector in $V$. Since $V$ is an $\Omega$-generated $\mathbb{C}$-graded vertex algebra we know that $v$ is a linear combination 
 $v=\sum_{j=1}^m c_j \tilde{v}^j$ where $c_j\in \mathbb{C}$ and each $ \tilde{v}^j$ is an
 element of the form $ v^k_{n_k}\cdots v^1_{n_1}u^0$ with  
 $n_1, \dots, n_r \in \mathbb{Z}$, and  $u^0 \in \Omega(V).$ If $deg(v)=0$ then we have that $\sum_{j=1}^m deg(\tilde{v}^j) =0$ where each $deg( \tilde{v}^j)$ satisfies either
 \[\deg( \tilde{v}^j)=0 \ \mbox{ or } \ Re(\deg(\tilde{v}^j))>0\] 
 by  Lemma \ref{grading-lemma}. Therefore, we obtain that $deg( \tilde{v}^j)=0$ for each $1\leq j\leq m$. By the argument above we have that each $\tilde{v}^j$ is an element in $\Omega(V)$ which implies that $v\in \Omega(V).$ 
\end{proof}

\begin{rmk} \label{rms} In \cite{LM} all  $\Omega$-generated $\mathbb{C}$-graded vertex algebras are assumed to be $\CR$-graded and referred to as $\mathbb{C}$-graded vertex algebras instead.
\end{rmk}

An  $\Omega$-generated $\mathbb{C}$-graded vertex algebra resembles a vertex operator algebra (with a possibly weaker non integer grading) in that it has a weight operator $L$ defined as in equation (\ref{Ldef}) which generalizes the zero Virasoro mode $L(0)$. Since we need to work in the $\mathbb{C}$-graded vertex algebra setting, we introduce the definition of a $\mathbb{C}$-graded conformal vertex algebra and show how it generalizes the concept of a conformal vertex  algebra.
\begin{dfn}\cite{HLZ} \label{ccva} A \textit{$\mathbb{C}$-graded conformal vertex algebra} $(V,Y, \vac, \omega)$ consists of a $\mathbb{C}$-graded vertex algebra
\[V=\bigoplus_{\lambda \in \mathbb{C}}V_{\lambda} \]
together with a distinguished vector $\omega \in V_2$ that satisfies the Virasoro relations: 

\vspace{0.2cm}
\noindent 
(i) $[L(n),L(m)]=(n-m)L(m+n)+\frac{1}{12}(n^3-n)\delta_{n,-m}c$ for $n,m \in \mathbb{Z}$, where $L(n)=:\omega_{n+1}$ for $n\in \mathbb{Z}$ and $c\in \mathbb{C}$ called the central charge of $V$, 

\noindent 
(ii) The $L(-1)$-derivative property: for any $v\in V$,
			$Y(L(-1)v,x)=\dfrac{d}{dx} Y(v,x),$

\noindent (iii) The $L(0)$-grading property: for $\mu \in \mathbb{C}$ and $v\in V_{\mu}$,
			    $L(0)v=\mu v=(wt \ v)v$.

\noindent A $\mathbb{Z}$-graded conformal vertex algebra is defined in the obvious way.	
\end{dfn}

\begin{dfn} A {\it vertex operator algebra} $(V,Y,\vac,\omega)$ is a $\mathbb{Z}$-graded conformal vertex algebra 
\[V=\bigoplus_{n \in \mathbb{Z}}V_{n}\]
such that \\
\noindent 
(i) $V_{n}=0$ for $n$ sufficiently negative, and

\noindent (ii) dim $V_{n}<\infty$ for $n\in \mathbb{Z}.$
\end{dfn}
{ Since the $\mathbb{Z}$-grading condition for a vertex operator algebra is too restrictive to work with the Weyl vertex algebras of all central charges, we will need the following modified concept of   an $\Omega$-generated $ \CR$-graded vertex operator algebra:}

\begin{dfn}\label{define-C-graded-VOA} A {\it  $\Omega$-generated $\CR$-graded vertex operator algebra}
is an $\Omega$-generated  $\CR$-graded vertex algebra $V=\bigoplus_{\lambda\in\mathbb{C}}V_{\lambda}$ that is also a $\mathbb{C}$-graded conformal vertex algebra with the following  additional properties:

\vspace{0.1cm}

\noindent (i) For $\lambda\in\mathbb{C}$, $V_{\lambda}=\{v\in V~|~L(0)v=\lambda v\} $ and $\dim V_{\lambda}<\infty$. 
\vspace{0.1cm}

\noindent (ii) $Re(\lambda)\geq | Im(\lambda)|$ for all but finitely many $\lambda\in Spec_VL(0)$.
\end{dfn}

\begin{rmk}
 Condition (ii) above, which may appear unnatural, guarantees that there are only finitely many eigenvalues $\lambda$ of $L(0)$ such that $Re(\lambda)<0$ and $V_\lambda\neq 0$. As explained in \cite{DM}, if an $\Omega$-generated $\CR$-graded vertex operator algebra is $\mathbb{R}$-graded (namely, if $V_{\lambda}\neq 0$ then $\lambda\in \mathbb{R}$) condition (ii) guarantees the usual lower boundedness condition that $V_r=0$ for all $r$ sufficiently negative.
\end{rmk}

\begin{rmk} Since $\omega\in V_2$, we can conclude that any  $\Omega$-generated $ \CR$-graded vertex operator algebra contains the vertex operator algebra generated by $\omega$.
\end{rmk}

The following are the relationships between the various types of vertex algebras introduced in this section: 
\begin{eqnarray*}
&&  \Omega VOA(\CR {\mathcal(V)})\subset \big( Conf(\mathcal{\mathbb{C}(V)})\cap  { \Omega(\CR(\mathcal{V}))} \big)\subset  \Omega(\mathcal{\mathbb{C}(V)})\subset \mathcal{\mathbb{C}(V)}\subset\mathcal{V},\\
&&\mathcal{VOA}\subset Conf(\mathcal{\mathbb{Z}(V)})\subset \mathcal{\mathbb{Z}(V)}\subset\mathcal{V}\\
\end{eqnarray*}
Here 
\begin{eqnarray*}
&&\mathcal{V}=\text{set of vertex algebras}\\
&&\mathcal{\mathbb{C}(V)}= \text{set of $\mathbb{C}$-graded vertex algebras}\\
&&\mathcal{\mathbb{Z}(V)}= \text{set of $\mathbb{Z}$-graded vertex algebras}\\
&& \Omega(\mathcal{\mathbb{C}(V)})= \text{set of  $\Omega$-generated  $\mathbb{C}$-graded vertex algebras}\\
&&\Omega(\CR(\mathcal{V}))= \text{set of  $\Omega$-generated  $\CR$-graded vertex algebras}\\
&&\Omega(\mathcal{\mathbb{Z}(V)})= \text{set of  $\Omega$-generated $\mathbb{Z}$-graded vertex algebras}\\
&&Conf(\mathcal{\mathbb{C}(V)})= \text{set of $\mathbb{C}$-graded conformal vertex algebras}\\
&&Conf(\mathcal{\mathbb{Z}(V)})= \text{set of $\mathbb{Z}$-graded conformal vertex algebras}\\
&&\mathcal{VOA}=\text{set of vertex operator algebras}\\
&& \Omega VOA( \CR\mathcal{(V)})= \text{set of  $\Omega$-generated  $\CR$-graded vertex operator algebras}\\
\end{eqnarray*}
\subsection{2.2. Modules for $\mathbb{C}$-graded vertex algebras}
Next, we  introduce various types of representations of $\mathbb{C}$-graded vertex algebras, again following or motivated by, for instance \cite{LL, HLZ, LM}. 
 \noindent We begin by recalling the definition of a weak $V$-module for a  fixed vertex  algebra $(V, Y, \mathbf{1})$ as  presented in \cite{LL}:

\begin{dfn} \label{weak}
 Let $V$ be a vertex algebra. A \textit{weak $V$-module} is a vector space $W$
		equipped with a vertex operator map
		\begin{equation*}
		\begin{aligned}
		Y_W \colon V &\rightarrow (\mathrm{End}\,W)[[x,x^{-1}]],\\
		v &\mapsto Y_W(v,x)=\sum_{n\in\mathbb{Z}} v_n^W\,x^{-n-1},
		\end{aligned}
		\end{equation*}
		satisfying the following axioms:
		\begin{enumerate}[label=\textup{(\roman*)}]
			\item
			The lower truncation condition: for  $v \in V$ and $w \in W$,  $Y_W(v,x)w$ has only finitely many
			terms with
			negative powers in $x$.
			
			\item
			The vacuum property: $Y_W(\mathbf{1},x)$ is the identity endomorphism $1_W$ of $W$.
			
			\item
			The Jacobi identity: for  $v_1,v_2\in V$,
\begin{multline*}
			x_0^{-1}\delta\left(\dfrac{x_1-x_2}{x_0}\right) Y_W(v_1,x_1)Y_W(v_2,x_2) 
			- x_0^{-1}\delta\left(\dfrac{-x_2+x_1}{x_0}\right)Y_W(v_2,x_2)Y_W(v_1,x_1)  \\
			= x_2^{-1}\delta\left(\dfrac{x_1-x_0}{x_2}\right)Y_W(Y(v_2,x_0)v_1,x_2).
			\end{multline*}
		\end{enumerate}
\end{dfn}

\begin{rmk}  In \cite{LL}, the notion of weak $V$-module {given above} is called  a $V$-module  for $V$ a vertex algebra, but if $V$ has, for instance the structure of a vertex operator algebra, then the structure $V$-module defined above is called in \cite{LL} a weak module for the vertex operator algebra structure of $V$.  Since we will mainly be concerned with extra ``vertex operator algebra"-type structures on $V$, to emphasize the differences between the weaker notion of a module for a vertex algebra versus a module for a vertex operator algebra, we have chosen to call these modules ``weak" throughout.
\end{rmk}
\begin{prop}\label{submodule} \cite{LL} Let $V$ be a vertex algebra and let $D$ be the linear map on $V$ given by $Dv=v_{-2}\mathbf{1}$  as in Remark \ref{D}. Let $W$ be a weak $V$-module. 
\begin{enumerate}
    \item Then \begin{equation*}
			Y_W(Dv,x)=\dfrac{d}{dx} Y_W(v,x).
			\end{equation*}
			\item Let $T$ be a subset of $W$ and let $\langle T\rangle$  denote the  submodule generated by $T$. Then \begin{equation*}
			    \langle T\rangle=Span\{v_nt~|~v\in V,~n\in\mathbb{Z}, t\in T\}.
			\end{equation*}
			\end{enumerate}
\end{prop}

Next, we recall the notion of a module over an $\Omega$-generated $\mathbb{C}$-graded vertex algebra $V$ as introduced in \cite{LM}.

\begin{dfn}\cite{LM} \label{C-graded-module-def}
Let $V$ be an $\Omega$-generated $ \CR$-graded vertex 
algebra. A  {\it $\mathbb{C}_{Re>0}$-graded  $V$-module} $ W$ is a weak $V$-module with
a grading of the form 
\[W=W(0) \bigoplus_{\substack{\tau\in\mathbb{C},\\ Re(\tau)>0}}   W(\tau)\] 
such that $W(0)\neq 0$ and for any homogeneous $v\in V_{\lambda}$, one has
$$v^W_nW(\tau)\subseteq W(\tau+\lambda-n-1).$$
 We say that a homogeneous element $w \in W(\tau)$ has {\it degree $\tau$. }
\end{dfn}

\begin{rmk}  In \cite{LM},  $\mathbb{C}_{Re>0}$-graded modules are referred to as admissible modules.
\end{rmk}

\begin{dfn}  Let $V$ be   an $\Omega$-generated $\CR$-graded vertex algebra and let \[W=W(0) \bigoplus_{\tau\in\mathbb{C}, Re(\tau)>0}W(\tau)\] 
be a   $\mathbb{C}_{Re>0}$-graded
$V$-module.  We define
\begin{align*}
\Omega(W) &= \{ w \in W \; | \; \mbox{for any $v \in V$ if
$v^W_nw\neq 0$ then either $|w|=|v^W_nw|$ or $Re(|w|)<Re(|v^W_nw|)$\} }\\
\end{align*}
\end{dfn}

Note in particular, that $W(0) \subset \Omega(W)$.  Moreover, $\Omega(W)$ consists of the vectors in $W$ that are annihilated by the action of any mode of $V$ that lowers the real part of its weight, similarly to  $\Omega(V)$ in Definition \ref{omegaV}.

 The following result was stated in \cite{LM} for  $\Omega$-generated $\mathbb{C}$-graded vertex algebras,  where it was assumed that the degree grading is well defined for these types of vertex algebras.  Here, we give the proof for the case in which $V$ is an $\Omega$-generated $\CR$-graded vertex algebra:

\begin{prop} [cf.\cite{LM}] \label{strongly-generated-LM-prop}
${}$
\begin{enumerate}
\item Any  $\Omega$-generated $\CR$-graded vertex algebra $V$ is  a $\mathbb{C}_{Re>0}$-graded $V$-module. 
\item If $W$ is a simple    $\mathbb{C}_{Re>0}$-graded $V$-module then  $\Omega(W)=W(0)$.  

\end{enumerate}
\end{prop}

 \begin{proof} The first statement follows directly from the degree grading in Definition \ref{degree} on $V$  together with Remark \ref{rmkVadmmod} and the definition of a $\mathbb{C}_{Re>0}$-graded $V$-module.  

 To prove the second statement, we first show that if $W=W(0)\bigoplus_{\tau\in\mathbb{C}, Re(\tau)>0}W(\tau)$ is a simple $\mathbb{C}_{Re>0}$-graded $V$-module then, $\Omega(W)\cap \left( \bigoplus_{\tau\in\mathbb{C}, Re(\tau)>0}W(\tau)\right)=0.$
To see this let $w\in\Omega(W)\cap (\bigoplus_{\tau\in\mathbb{C}, Re(\tau)>0}W(\tau))$. Then $\langle w\rangle =Span\{v_n^W w~|~v\in V, n\in\mathbb{Z}\}\subset \bigoplus_{\tau\in\mathbb{C}, Re(\tau)>0}W_{\tau}$ because $w\in \Omega(W)$ and so in particular $Re(|v_n^Ww|)\geq Re(|w|)>0$ for every $v \in V, n \in \mathbb{Z}$ such that $v_n^Ww\neq 0$.  Since $\langle w\rangle$ is a proper $V$-submodule $\langle w\rangle \varsubsetneq W$, we can conclude that $\langle w \rangle=\{0\}$. In particular, $w=0$ and we have shown that $\Omega(W)\cap \left( \bigoplus_{\tau\in\mathbb{C}, Re(\tau)>0}W(\tau)\right)=0.$ 

Finally, we show that $\Omega(W)=W(0)$. Let $u\in \Omega(W)$. Since $u\in W$ we can write $u=w'+w''$ for $w'\in W(0)$ and $w''\in \bigoplus_{\tau\in\mathbb{C}, Re(\tau)>0}W(\tau)$. Since $w''=u-w'$ and $W(0)\subseteq \Omega(W)$, we can conclude that $w''\in \Omega(W)$. Moreover, $w''\in \Omega(W)\cap\bigoplus_{\tau\in\mathbb{C}, Re(\tau)>0}W(\tau)$ which by our previous argument is $0$. This implies that $w''=0$ and $u=w' \in W(0)$. Hence $\Omega(W)=W(0)$. 
\end{proof}

\begin{dfn}\label{ordinary module} Let $V$ be an $\Omega$-generated $ \CR$-graded vertex operator algebra. An {\it ordinary $V$-module} $W$ is a weak $V$-module that admits a decomposition into generalized eigenspaces via the spectrum of $L_W(0)$ as follows:

\noindent (i) $W=\bigoplus_{\lambda\in\mathbb{C}}W(\lambda)$ where $W(\lambda)=\{w\in W~|~L_W(0)w=\lambda w\}$. 

\noindent (ii) $\dim W(\lambda)<\infty$ for all $\lambda\in\mathbb{C}$.

\noindent (iii) $Re(\lambda)>0$ for all but finitely many $\lambda\in Spec L_W(0)$.
\end{dfn}
 Finally, we introduce the notion of rationality for the representations of an $\Omega$-generated $ \CR$-graded vertex operator algebra:
\begin{dfn} Let $V$ be an $\Omega$-generated $ \CR$-graded vertex operator algebra. $V$ is called {\em rational} if the category of    $\mathbb{C}_{Re>0}$-graded $V$-modules is semisimple,  i.e. every $\mathbb{C}_{Re>0}$-graded $V$-module is completely reducible, i.e. the sum of simple $ \CR$ modules.
\end{dfn}
\section{3. The Weyl vertex algebra: Classification of its $\mathbb{C}$-graded structures} \label{weyl-section}

In this section, we introduce the rank one Weyl vertex algebra, denoted $M$, with a family of conformal elements $\omega_\mu$ parameterized by $\mu \in  \mathbb{C}$, following, for instance \cite{AP} (see also \cite{L}).  We denote $M$ with the conformal structure by $( _\mu M, \omega_\mu)$ or just $_\mu M$.  We discuss the various gradings and associated refined vertex algebra structures imposed on the rank one Weyl vertex algebra $M$ by the choice of $\mu$.  The rank $n$ Weyl vertex algebra, for $n \in \mathbb{Z}_+$, is then the $n$-fold tensor product of $M$.

\begin{dfn}

Let $\mathcal{L}$ be the infinite-dimensional Lie algebra with generators $K, a(m)$, and $a^*(n)$ with $m,n \in \mathbb{Z}$ such that $K$ is in the center and the bracket is given by
\[ \left[ a(m), a^* (n) \right] = \delta_{m+n,0} K.\]
We define  the {\em  rank one Weyl algebra}  $\mathcal{A}_1$ to be the quotient
$$\mathcal{A}_1= \frac{\mathcal{U}\mathcal{(L)}}{\langle K-1\rangle },$$
where  $\mathcal{U} \mathcal{(L)}$ denotes the  universal enveloping algebra of $\mathcal{L}$ and $\langle K - 1\rangle$ is the two sided ideal generated by $K - 1.$ 

  We have that  $\mathcal{A}_1$  is an associative algebra with generators $a(m)$, $a^*(n)$,  for $m, n\in\mathbb{Z}$, and relations
\begin{eqnarray} \label{com}
[a(m),a^*(n)] & = & \delta_{m+n,0}  \\
 \left[a(m),a(n) \right] & = & [a^*(m),a^*(n)] \ = \ 0
\end{eqnarray}
for all $m,n\in\mathbb{Z}$. 
\end{dfn}

The Weyl algebra $\mathcal{A}_1$ has a countably infinite family of automorphisms, called {\it spectral flow} automorphisms given by 
\begin{equation}\label{Weyl-auto}
\rho_s: \mathcal{A}_1  \longrightarrow  \mathcal{A}_1, \quad a(n)  \mapsto  a(n+s), \quad a^*(n)  \mapsto  a^*(n-s),
 \end{equation}
for $s \in \mathbb{Z}$, as well as the automorphism
\begin{equation}
\varphi_t : \mathcal{A}_1  \longrightarrow  \mathcal{A}_1, \quad a(n)  \mapsto  ta^*(n), \quad a^*(n)  \mapsto  -t^{-1}a(n),
\end{equation}
for $t\in \mathbb{C}^\times$.

The  (rank one) Weyl vertex algebra  $M$ can be realized as an induced module for the Lie algebra $\mathcal{L}$ as follows. We first fix a triangular decomposition  of $\mathcal{L}=\mathcal{L}^-\oplus\mathcal{L}^0\oplus\mathcal{L}^+ $ where
\begin{align*}
 \mathcal{L}^{-}&= \mathrm{span}_{\mathbb{C}}\{ a(-n), a^* (-m) \, | \, n\geq1 ,  m \geq 0 \} \\
 \mathcal{L}^{0}&= \mathrm{span}_{\mathbb{C}}\{  K \} \\
 \mathcal{L}^{+}&= \mathrm{span}_{\mathbb{C}}\{ a(n), a^* (m+1) \in \mathcal{L} \, | \, n \geq 0, m\geq 0 \}. 
\end{align*}
(see for instance \cite{RW} where this is called the normal triangular decomposition).
Next, we give the one dimensional vector space $\mathbb{C} \bf 1$  the $\mathcal{L}^{0}\oplus \mathcal{L}^{+}$-module structure
 given by
\begin{align*}
a(0){\bf 1}&=0\\
K\vac&=\vac\\
a(n){\bf 1}&=0 \text{ for }n > 0, \\
a^*(m+1){\bf 1}&=0  \text{ for }m \geq 0, 
\end{align*}
 and define $M$ to be  the induced module $$M = \mathcal{U}(\mathcal{L}) \otimes_{\mathcal{U}(\mathcal{L}^0\oplus \mathcal{L}^+ )} \mathbb C \bf 1. $$
\noindent Then, $M$ is a simple Weyl module and, as a vector space,  $ M\cong \mathbb{C}[a(-n),a^*(-m)~|~n>0, ~m\geq 0]$. 
There is a unique vertex algebra structure on  $M$,  (see for instance Theorem 5.7.1 in \cite{LL} or Lemma 11.3.8 in \cite{FBZ})  given by $( M,Y,{\bf 1})$ with vertex operator map $Y:  M \rightarrow \End(M)[[z,z^{-1}]]$ such that 
\begin{eqnarray}
&Y(a(-1){\bf 1},z)=a(z),& Y(a^*(0){\bf 1},z)=a^*(z), \nonumber \\
&a(z)=\sum_{n\in\mathbb{Z}}a(n)z^{-n-1},& a^*(z)=\sum_{n\in\mathbb{Z}}a^*(n)z^{-n}\label{genweyl}.
\end{eqnarray}
In particular, $$Y(a(-1)a^*(0){\bf 1},z)=:a(z)a^*(z): $$
 where $:a(z)a^*(z):$ denotes the ordered product of the fields $a(z)$ and $a^*(z)$ given by
\begin{align*}
 :a(z)a^*(z):=a(z)^+a^*(z)+a^*(z)a(z)^-,
\end{align*}
with $a(z)^+=\sum_{n\leq -1}a(n)z^{-n-1}$, $a(z)^-=\sum_{n\geq 0}a(n)z^{-n-1}$.

In terms of the operator product expansion of the vertex operators, i.e. the corresponding fields, we have 
\[ a(z) a^*(w) = \frac{1}{z-w} +  : a(z) a^*(w) :\]

Moreover, the map $Y:  M \rightarrow \End(M)[[z,z^{-1}]]$ is given by
\begin{align*}
Y(a(-m_1-1)a(-m_2-1)& \dots a(-m_k-1)a^*(-n_1)\dots a^*(-n_l)\vac, z)\\
& = \prod_{i=1}^k \frac{1}{m_i!} \prod_{j=1}^l \frac{1}{n_j!} :\partial^{m_1}a(z)\cdots  \partial^{m_k}a(z) \partial^{n_1}a^*(z)\cdots\partial^{n_l}a^*(z):  
\end{align*}
for $n_1, \dots n_k, m_1, \dots m_l \in \mathbb{Z}_{\geq 0}.$

\begin{rmk} ${}$ \label{bgrem}

\begin{enumerate}
    \item The fields $a(z)$ and $a^*(z)$ defined in (\ref{genweyl}) are usually denoted by $\beta(z)$ and $\gamma(z)$ in the Physics literature (up to a choice of sign) where the vertex algebra $M$ is referred to as the $\beta\gamma$ vertex algebra,  or $\beta \gamma$-system.

    \item Since for all $n\in \mathbb{Z}$, the $n$ modes of the fields $Y(a(-1)\vac, z)=a(z), Y(a^*(0)\vac)= a^*(z)$ satisfy 
\begin{align*}
    &(a(-1)\vac)_n=a(n)\\
    &(a^*(0)\vac)_n=a^*(n+1),
\end{align*}
we have that the set $T=\{a(-1)\vac, a^*(0)\vac\}$ is a set of strong generators for the vertex algebra $M$ in the sense of Definition \ref{sgen}. Namely, $M$ is spanned by the set of normally ordered monomials 
$$\{ :\partial^{k_1} \alpha^{i_1} \dots \partial^{k_l} \alpha^{i_k}: | \  k_1, \dots, k_l \geq 0, \  \alpha^{i_j}\in  T \}.$$ 
 Therefore, $M$ is  strongly finitely generated as a vertex algebra in the sense of Definition \ref{sgen}.
\end{enumerate}
\end{rmk}

From the simple relations between the modes of the strong generators $a(-1)\vac$ and $a(0)\vac$ given by (\ref{com}) together with Remark \ref{bgrem}, it is easy to see that $M$ is a simple vertex algebra.

Let $\beta:=a(-1)a^*(0){\bf 1}$. We set $\beta(z)=Y(\beta,z)=\sum_{n\in\mathbb{Z}}\beta(n)z^{-n-1}$. (We note that in \cite{AP}, there was a typo in the exponent of $z$ in the expansion of $\beta(z)$.)
 We note in particular, that in this notation $$\beta(-2)\vac=a(-2)a^*(0){\bf 1}+a(-1)a^*(-1){\bf 1}. $$
Then $\beta$ is a Heisenberg vector in $M$ of level $-1$. Namely,  for $n,m\in\mathbb{Z}$, we have 
$$[\beta(m),\beta(n)]=-m\delta_{m+n,0}. $$ as operators on $M$, and therefore
$$\beta(z)\beta(w)=-\frac{1}{(z-w)^2}+:\beta(z)\beta(w):$$ In addition, we have 
$$[\beta(m),a(n)]=-a(m+n), \quad \mbox{ and} \quad [\beta(m),a^*(n)]=a^*(m+n).$$

 We are interested in the possible $\mathbb{C}$-graded conformal vertex algebra structures on the vertex algebra $M$.  The vertex algebra $M$ admits a family of Virasoro vectors 
\begin{eqnarray} \label{conformal}
\omega_{\mu} &=& (1-\mu)a(-1)a^*(-1){\bf 1}-\mu a(-2)a^*(0){\bf 1} \qquad  \mbox{for $\mu\in\mathbb{C}$,} \nonumber \\
&=&  a(-1)a^*(-1){\bf 1}-\mu \left(a(-1)a^*(-1){\bf 1}+ a(-2)a^*(0){\bf 1}\right) \nonumber\\
&=&  a(-1)a^*(-1){\bf 1}-\mu \beta(-2)\vac,
\end{eqnarray}
of central charge 
\begin{equation} \label{central}
c_{\mu}=2(6\mu(\mu-1)+1) .   
\end{equation} 
 The corresponding Virasoro field is 
\begin{align}
L^{\mu}(z)=(1-\mu):a(z)\partial a^*(z):-\mu :\partial a(z) a^*(z):    
\end{align}
and it satisfies
\begin{align*}
    L^{\mu}(z) L^{\mu}(w)= \frac{1-6\mu+6\mu^2}{(z-w)^4}+\frac{2L^{\mu}(w)}{(z-w)^2}+\frac{\partial_w L^{\mu}(w)}{z-w}+: L^{\mu}(z) L^{\mu}(w):
\end{align*}

 This gives a $\mathbb{C}$ grading on $M$ as we give explicitly below, and we denote the particular $\mathbb{C}$-graded conformal vertex algebra structure on $M$ by
\[(_\mu M, Y, {\bf 1}, \omega_\mu), \]
or just $_\mu M$.

\begin{lem}\label{Weyl-iso-lemma}
The composition of the spectral flow $\rho_1$ and $\varphi_1$ automorphisms  of the $Weyl$ algebra lifts to give the following isomorphisms of $\mathbb{C}$-graded conformal vertex algebras: 
\begin{equation}
\varphi_1 \circ \rho_1 : (_\mu M, \omega_\mu) \stackrel{\cong}{\longrightarrow} (_{1 - \mu} M, \omega_{1 - \mu})
\end{equation}
given explicitly on $M$ by
\begin{multline*}
a(-m_1-1) \cdots a(-m_k-1)a^*(-n_1) \cdots a^*(-n_l) {\bf 1} \mapsto \\
(-1)^l  a(-n_1 -1) \cdots a(-n_l - 1) a^*(-m_1) \cdots a^*(-m_k) {\bf 1},
\end{multline*}
for $k,l \in \mathbb{N}$ and $m_i, n_j \in \mathbb{N}$. Or more generally, for the vertex algebra structure, letting $F = \varphi_1 \circ \rho_1$, we define
\begin{equation}\label{iso-structure}
F(u^1_{n_1}\cdots u^k_{n_k} {\bf 1})= [F(u^1)]_{n_1} \cdots   [F(u^k)]_{n_k}{\bf 1}
\end{equation}
for $u^j = a(-1)\vac$ or $a^*(0)\vac$ for $j = 1, \dots, k$, and $n_1, \dots, n_k \in \mathbb{Z}$.

Moreover, this is the only $\mathbb{C}$-graded conformal vertex algebra isomorphism between the $(_\mu M, \omega_\mu)$, for  distinct $\mu \in\mathbb{C}$.  In particular, the central charge $c_\mu = c_{1-\mu}$ completely determines $(_\mu M, \omega_\mu)$ up to isomorphism.  
\end{lem}

\begin{proof} 
 By definition, $F = \varphi_1 \circ \rho_1$ is a vector space isomorphism.  Eqn.\ (\ref{iso-structure}) implies $F$ is a vertex algebra homomorphism, as follows: By the definition of $F$, we have $F(u_nv) = F(u_n v_{-1} {\bf 1}) = F(u)_nF(v)_{-1} {\bf 1} = F(u)_n F(v)$ for $u, v \in \{a(-1){\bf 1}, a^*(0) {\bf 1} \}$. By induction on $k$ we have that $F(u_nv) = F(u)_nF(v)$ for $v = u_{n_1}^1 \cdots u_{n_k}^k {\bf 1}$ for $u^1, \dots, u^k \in \{a(-1){\bf 1}, a^*(0) {\bf 1} \}$ and $n_1, \dots, n_k \in \mathbb{Z}$.

Then note that 
\begin{eqnarray*}
F(Y(a(-1) {\bf 1}, z)v) &=& F\left( \sum_{n \in \mathbb{Z}} a(n)v z^{-n-1}\right) \ = \ \sum_{n\in \mathbb{Z}} F(a(n)v) z^{-n-1} \\
&=& \sum_{n \in \mathbb{Z}} F((a(-1){\bf 1})_n v) z^{-n-1} \ = \ \sum_{n \in \mathbb{Z}} [F(a(-1){\bf 1})]_n F(v) z^{-n-1} \\
&=& \sum_{n \in \mathbb{Z}} (a^*(0) {\bf 1})_n F(v) z^{-n-1} \ = \ \sum_{j \in \mathbb{Z}} a^*(j) z^{-j} F(v) \\
&=& Y(a^*(0){\bf 1},z) F(v) \ = \ Y(F(a(-1) {\bf 1}), z ) F(v) ,
\end{eqnarray*}
and 
\begin{eqnarray*}
F(Y(a^*(0) {\bf 1}, z)v) &=& F\left( \sum_{n \in \mathbb{Z}} a^*(n)v z^{-n}\right) \ = \ \sum_{n\in \mathbb{Z}} F(a^*(n)v) z^{-n} \\
&=& \sum_{n \in \mathbb{Z}} F((a^*(0){\bf 1})_{n-1} v) z^{-n} \ = \sum_{n \in \mathbb{Z}} [F(a^*(0) {\bf 1})]_{n-1} F(v) z^{-n} \\
&=&  - \sum_{n \in \mathbb{Z}}(a(-1){\bf 1})_{n-1} F(v) z^{-n} \ = \ - \sum_{ j \in \mathbb{Z}} (a(-1) {\bf 1})_j F(v) z^{-j-1}\\
&=& - Y(a(-1){\bf 1},z) F(v) \ = \ Y(F(a^*(0) {\bf 1}), z ) F(v) .
\end{eqnarray*}
Therefore by Proposition 5.7.9 in \cite{LL}, we have $F(Y(u,z)v) = Y(F(u),z)F(v)$ for all $u,v \in M$, and $F$ is a homomorphism of vertex algebras.  Since it is a bijection, it is an isomorphism of vertex algebras.

 Finally, 
for $\mu \in \mathbb{C}$
\begin{eqnarray*}
\varphi_1 \circ \rho_1(\omega_\mu) &=&
\varphi_1 \circ \rho_1( (1-\mu) a(-1) a^*(-1) {\bf 1} - \mu a(-2) a^*(0) {\bf 1}) \\
&=& \varphi_1( (1-\mu) a(0) a^*(-2) {\bf 1} - \mu a(-1) a^*(-1) {\bf 1}) \\
&=& - (1-\mu) a^*(0) a(-2) {\bf 1} + \mu a^*(-1) a(-1) {\bf 1} \\
&=& \mu a(-1) a^*(-1) {\bf 1} - (1-\mu) a^*(0) a(-2) {\bf 1}\\
&=& (1 - (1-\mu))a(-1) a^*(-1) {\bf 1} - (1-\mu) a(-2) a^*(0)  {\bf 1}\\
&=& \omega_{1 - \mu},
\end{eqnarray*}
proving that this is a isomorphism of $\mathbb{C}$-graded conformal vertex algebras.  Then since  $c_\mu = 2(6\mu(\mu - 1) + 1) = 2(6 \nu (\nu - 1) + 1) = c_\nu$ for $\mu \neq \nu$ implies $\nu = 1-\mu$, this shows these are the only isomorphisms  between the conformal structures on $M$. 
\end{proof}

Let $$L^{\mu}(z)=Y(\omega_{\mu},z)=\sum_{n\in\mathbb{Z}}L^{\mu}(n)z^{-n-2}.$$ 
For $\mu=0$, we set  $\omega:=\omega_0$, $L(n):=L^{0}(n)$, and then $c_{0}=2$.  More generally, we have that for $\mu \in \mathbb{C}$,  
\begin{equation} \label{eq13old}
\omega_{\mu}=\omega-\mu \beta(-2){\bf 1}.
\end{equation}
Furthermore, since $(\beta(-2){\bf 1})_0=(D\beta)_0=0$,  and $(\beta(-2){\bf 1})_1=(D\beta)_1=-\beta(0)$ where $D$ is the endomorphism described in Remark \ref{D} we thus have that 
\begin{eqnarray*}
L^{\mu}(-1) &=& L(-1) \qquad \text{ for all }\mu\in\mathbb{C}, \\ 
L^{\mu}(0) &=& L(0) +  \mu \beta(0). 
\end{eqnarray*}
In addition, for all $m,n \in \mathbb{Z}$
\begin{eqnarray*}
\left[L(m),a(n) \right] & = & -na(m+n),\\ 
\left[ L(m),a^*(n) \right] & = & -(m+n)a^*(m+n).
\end{eqnarray*}
In particular, we have 
\begin{eqnarray*}
[L(0),a(n)] &=&-na(n),\\
\left[L(0),a^*(n)\right] &=&-na^*(n),\\
\left[  L^{\mu}(0)  ,a(n) \right]&=& [L(0)  +  \mu\beta(0),a(n)] \ = \ -na(n)  -  \mu a(n) \ =\ (-n  -  \mu)a(n)\\
\left[ L^{\mu}(0)  ,a^*(n) \right] &=&[L(0)  +  \mu\beta(0), a^*(n)] \ = \ -na^*(n) +  \mu a^*(n) \ = \ (-n  +  \mu)a^*(n).
\end{eqnarray*}

Notice that for integers $m_1\geq \cdots \geq m_k\geq 0$, $n_1\geq \cdots \geq n_t\geq 0$,  and $k,t \in \mathbb{Z}_+$, we have 
\begin{eqnarray*}
L^{\mu}(0)a(-m_1-1)\cdots a(-m_k-1){\bf 1}&=&((m_1+ \cdots +m_k+k)-k\mu)a(-m_1-1) \cdots a(-m_k-1){\bf 1}\\
&=&((m_1+ \cdots +m_k)+k(1-\mu))a(-m_1-1) \cdots a(-m_k-1){\bf 1}, \\ 
L^{\mu}(0)a^*(-n_1)\cdots a^*(-n_t){\bf 1}& =&  ((n_1+\cdots +n_t)+t\mu)a^*(-n_1)\cdots a^*(-n_t){\bf 1},
\end{eqnarray*}
and 
\begin{eqnarray}\label{L(0)-grading}
\lefteqn{L^{\mu}(0)a(-m_1-1) \cdots a(-m_k-1)a^*(-n_1) \cdots a^*(-n_t){\bf 1}}\\
&=&  \Big( \sum_{i=1}^km_i+\sum_{j=1}^tn_j+k(1-\mu)+t\mu \Big)  a(-m_1-1) \cdots a(-m_k-1)a^*(-n_1)  \cdots a^*(-n_t){\bf 1} .\nonumber
\end{eqnarray}

 Thus from Eqn.\ (\ref{L(0)-grading}), an element $v = a(-m_1-1) \cdots a(-m_k-1)a^*(-n_1) \cdots a^*(-n_t){\bf 1} \in \, _\mu M$ has an $L^{\mu}(0)$-grading of the form 
\begin{equation}\label{grading}
\mathrm{wt} \, v = r + s + k(1 - \mu) + t\mu  \qquad \mbox{ for $r,s,k, t \in \mathbb{N}$, with $k,t\in \mathbb{Z}_+$ if $r,s \in \mathbb{Z}_+$, respectively.} 
\end{equation} 
That is the action of $L^{\mu}(0)$ on $_\mu M$ defines a $\mathbb{C}$-grading on $_\mu M$ which gives $_\mu M$ the structure of a $\mathbb{C}$-graded vertex algebra.  It is also a $\mathbb{C}$-graded conformal vertex algebra  with strong generators $a(-1)\vac$, $a^*(0)\vac$ which satisfy
\begin{align} \label{deggen1}
&|a(-1)\vac|=1-\mu \\
&|a^*(0)\vac|=\mu\label{deggen2}
\end{align}
However, for only certain values of $\mu$ is $_\mu M$  an $\Omega$-generated  $\CR$-graded vertex operator algebra in the sense of Definition \ref{define-C-graded-VOA}.

We are interested in these values of $\mu$ which give $_\mu M$  an $\Omega$-generated $ \CR$-graded vertex operator algebra structure and the nature of the representations of these vertex algebras $_\mu M$. 

To that end we note that Eqn.\ (\ref{grading}) implies that 
\begin{equation}\label{grading2}
\mathrm{wt} \, v = r + s + k + \mathrm{Re}(\mu)(t-k) + i \mathrm{Im}(\mu)(t-k) 
\end{equation}
for $r,s,k, t \in \mathbb{N}$ with $k,t\in \mathbb{Z}_+$ if $r,s \in \mathbb{Z}_+$, respectively. 

The analysis of the structure of $_\mu M$ naturally falls into the following 5 cases:

{\bf Case 1:}   $\mu=0, 1$, i.e., $c = 2$.  Then with respect to $L^\mu(0)$ weight grading, we the  have $(_0 M=\bigoplus_{n=0}^{\infty} \, _0 M_n,\omega_0)  = (_0 M,\omega)$ is an $\mathbb{N}$-graded conformal vertex algebra with central charge 2. Moreover,  the space of vectors of $L^\mu(0) = L(0)$-weight zero is equal to $\Omega (_0 M)$ and is given by
\[_0 M_0 = \mathrm{Span} \{a^*(0)\cdots a^*(0){\bf 1} = a^*(0)^t {\bf 1} \, | \, t \in \mathbb{N}\} 
\]
which is an infinite-dimensional subspace of $_0 M$.  
Analogously 
\[ \Omega( _1 M) =\  _1 M_0 = \mathrm{Span} \{a(-1)\cdots a(-1){\bf 1} = a(-1)^k {\bf 1} \, | \, k \in \mathbb{N} \} .\] 

$_0 M$ is not a vertex operator algebra (or for that matter,  an $\Omega$-generated $\mathbb{C}$-graded vertex operator algebra) since it has infinite-dimensional weight spaces. 

$_0 M$ is  an $\Omega$-generated $ \CR$-graded vertex algebra, and the $L(0)$-weight spaces of $_0 M$ are also the degree spaces of $_0 M$ viewed either as  an $\Omega$-generated $ \CR$-graded vertex algebra, or as a $\mathbb{C}_{Re>0}$-graded module over itself. 

$_0 M$ is referred to as the  Weyl vertex algebra with central charge 2 and is the unique rank 1 Weyl conformal vertex algebra with central charge 2 up to isomorphism  by Lemma \ref{Weyl-iso-lemma}.

{\bf  Case 2:}  $\mu \in \mathbb{R}$ and $0<\mu<1$,  i.e., $c \in \mathbb{R}$, and   $-1<c<2$.   In this case we have $0<\mu=Re(\mu)<1$ and $0< Re(1-\mu)<1$, and so from Eqn.\ (\ref{grading}) we have that $_\mu M$  (or equivalently $_{1-\mu} M$) is an $\mathbb{R}$-graded conformal vertex algebra, and is  $\Omega$-generated with  $\Omega(_\mu M) = \, _\mu M_0 =\mathbb{C}{\bf 1}$, and 
 
\[_\mu M = \bigoplus_{\substack{\lambda \in \mathbb{R} \\ \lambda\geq 0}} \, _\mu M_\lambda = \,  _\mu M(0)   \bigoplus_{\substack{\lambda \in \mathbb{R} \\ \lambda> 0}} \, _\mu M(\lambda); \]
where the degree spaces and weight spaces coincide.  Furthermore $_\mu M_\lambda = \, _\mu M(\lambda) = 0$ unless $\lambda \in (\mathbb{N} + \mu \mathbb{Z})\cap \mathbb{R}_+$, and in fact 
\begin{eqnarray*}
Spec_{_\mu M} L^{\mu} (0) &=& \{ r + s + k + \mu(t-k) \; | \; r,s, k, t \in \mathbb{N}, \mbox{and $k,t \in \mathbb{Z}_+$ if $r,s \in \mathbb{Z}_+$, resp.} \} \\
&=& \mu \mathbb{N} + (1-\mu)\mathbb{N}.
\end{eqnarray*}

In this case, we have $\dim \, _\mu M_\lambda < \infty$, since $_\mu M$ is  $\Omega$-generated by the finite dimensional set $\Omega(_\mu M) = \, _\mu M_0 =\mathbb{C}{\bf 1}$, and the generating set $S = \{ a(-1) {\bf 1}, a^*(0) {\bf 1}\}$ of positive non-integral weights $\mu$ and $1-\mu$, respectively, between 0 and 1.

Finally, noting that  $Re(\lambda) \geq 0 = |Im(\lambda)|$ for all $\lambda \in Spec_{\mu M} L^\mu(0)$, we conclude that $_\mu M$ is  an $\Omega$-generated $ \CR$-graded vertex operator algebra. 

 {\bf Case 3:}  $\mu \in \mathbb{C}$, $Im(\mu)\neq 0$ and $Re(\mu) = 0$ or $1$.  Since $_\mu M \cong _{1-\mu} M$, without loss of generality we may assume $Re(\mu) =0$.   Setting $\mu=iq$, then $1-\mu= 1 - iq$, from Eqn.\ (\ref{grading}), we have that $_\mu M_0 = \mathbb{C}{\bf 1}$, and with respect to the weight grading given by $L^\mu(0)$ and denoted $|u|$, we have that $Re(|u|)>0$ unless $Re(|u|) = 0$, in which case $u \in \mathrm{Span} \{a^*(0)\cdots a^*(0){\bf 1} = a^*(0)^t {\bf 1} \, | \, t \in \mathbb{N} \}$.  However, in this case $|a^*(0)^t {\bf 1}| = tiq$.  Thus $_\mu M_0 = \mathbb{C} {\bf 1} = \Omega (_\mu M)$ and the degree grading and $L^\mu(0)$-weight grading coincide.  Furthermore $\dim ( _\mu M_\lambda) <\infty$ for each $\lambda \in \mathbb{C}$.  
 
 Thus in this case, $_\mu M$ is  an $\Omega$-generated  $\CR$-graded vertex algebra and therefore a $\mathbb{C}_{Re>0}$-graded module over itself. However as a module over itself, $_\mu M$ has an infinite number of $\lambda \in Spec_{_\mu M}L^\mu(0)$ with $Re(\lambda)=0$ as we now show below. 
 
 Since the weight spaces of $_\mu M$ in this case are finite dimensional, one might think it is a candidate for  an $\Omega$-generated $ \CR$-graded vertex operator algebra.  However,  the question is: Does it satisfy $Re(\lambda) \geq |Im(\lambda)|$ for all  but fintely many $\lambda \in Spec_{_\mu M} L^\mu(0)$? And here the answer is no, since  $|a^*(0) {\bf 1} | = tiq  = \lambda$ implies $Re(\lambda) = 0 < t|q| = |Im(\lambda)|$ for $t \neq 0$.  Thus $_\mu M$ for $\mu \in i \mathbb{R}$ is an example of  an $\Omega$-generated $\mathbb{C}$-graded conformal vertex algebra that is not  an $\Omega$-generated $\mathbb{C}$-graded vertex operator algebra even though $\dim (_\mu M_\lambda) < \infty$. 
 
Analogous results hold for $\mu \in \mathbb{C}$ with $Re(\mu) = 1$ by Lemma \ref{Weyl-iso-lemma}.

{\bf Case 4:}   $\mu \in \mathbb{C}$, $Im(\mu)\neq 0$ and $0<Re(\mu)<1$. Setting $\mu=p+iq$, then $1-\mu=(1-p)+i(-q)$ and  $Re(1 - \mu) = 1-p$ satisfies $0< Re(1-\mu) < 1$.  Thus from Eqn.\ (\ref{grading}), we have that $_\mu M_0 = \mathbb{C}{\bf 1}$, and with respect to the weight grading given by $L^\mu(0)$ and denoted $|u|$, we have that $Re(|u|)>0$ unless $Re(|u|) = 0$, in which case $u \in \mathbb{C}{\bf 1}$.  Thus  
$_\mu M$ is $\mathbb{C}_{Re>0}$-graded with $\Omega (_\mu M)=\mathbb{C}{\bf 1}$.  Therefore  $_\mu M$ is  an $\Omega$-generated $ \CR$-graded vertex algebra and the $L^\mu(0)$-weight spaces correspond with the degree spaces.
 More precisely
\begin{eqnarray*}
Spec_{_\mu M} L^{\mu} (0) &=& \{ r + s + k + p(t-k) + iq(t-k) \; | \; r,s, k, t \in \mathbb{N}, \mbox{and $k,t \in \mathbb{Z}_+$ if $r,s \in \mathbb{Z}_+$, resp.} \} \\
&\subset& p \mathbb{N} + (1- p)\mathbb{N} + iq\mathbb{Z}.
\end{eqnarray*}

In this case, we also have that  $\dim V_\lambda <\infty$ for all $\lambda \in Spec_{_\mu M} L^\mu (0)$.  To see this,  we observe that if we consider only the real part of the weight grading for $_\mu M$, then the grading is the same as that for Case 2.  That is for $v \in M$ with $L^\mu (0) v = \lambda v$, then $L^{Re(\mu)} (0) v = Re(\lambda) v$.  Thus $\dim _\mu M_\lambda \leq \dim \  _{Re(\mu)} M_{Re(\lambda)} <\infty$.  

To analyze when $_\mu M$ is also  an $\Omega$-generated $ \CR$-graded vertex operator algebra, we need to determine if  $Re(\lambda) \geq |Im(\lambda)|$ for all but finitely many weights $\lambda \in Spec_{_\mu M} L^\mu (0)$.

Here we split into two subcases:

{\bf Case 4(a):}  If either $0< Re(\mu) \leq 1/2$ and $|Im(\mu)| \leq Re(\mu)$, or $0 < Re(1-\mu) \leq 1/2$ and $|Im(\mu)| \leq Re(1-\mu)$ hold, then we claim that $_\mu M$ is  an $\Omega$-generated $ \CR$-graded vertex operator algebra. We first prove this for $0< Re(\mu) \leq 1/2$ and $|Im(\mu)| \leq Re(\mu)$ and then note that since $_\mu M \cong \,  _{1-\mu} M$, the result will hold for $0 < Re(1-\mu) \leq 1/2$ and $|Im(\mu)| \leq Re(1-\mu)$.  

So assume $0< Re(\mu) \leq 1/2$ and $|Im(\mu)| \leq Re(\mu)$, i.e. writing $\mu = p + iq$, we have $0< p \leq 1/2$ and $|q|<p$.   Then for any $\lambda = r+s + k + p(t-k) + iq(t-k)  \in Spec_{_\mu M} L^\mu (0)$, we have that $|Im(\lambda)| = |q(t-k)|\leq |p(t-k)| \leq r+ s + k + |p(t-k)|$, whereas $|Re(\lambda)| = |r+s+ k + p(t-k)|$.  Thus if $t-k\geq 0$, we have $|Im(\lambda)| \leq |Re(\lambda)|$.  If $t-k<0$ then $k\neq 0$ and $Re(\lambda) \geq k + p(t-k)= k(1-p) + pt$. But  since $(1-p)\geq p$, we have $Re(\lambda) \geq kp + pt = p(k+t)$ and thus $|Im(\mu)| = |q(t-k)| \leq |p(t-k)| \leq |p(t+k)| \leq |Re(\lambda)|$. Therefore $Re(\lambda) \geq |Im(\lambda)|$ for all weights $\lambda \in Spec_{_\mu M} L^\mu (0)$.  

Therefore we have that in this case $_\mu M$ is  an $\Omega$-generated $ \CR$-graded vertex operator algebra.

{\bf Case 4(b):}  If $0<Re(\mu) \leq 1/2$ and $|Im(\mu)|> Re(\mu)$, or if $0< Re(1-\mu) <1/2$ and $|Im(\mu)|>Re(1-\mu)$, then we claim that $_\mu M$ is not  an $\Omega$-generated $ \CR$-graded vertex operator algebra.  To see this, we first prove the result for $0<Re(\mu) \leq 1/2$, but $|Im(\mu)|> Re(\mu)$, and then note that since $_\mu M \cong \, _{1-\mu}M$, the resut will hold for $0< Re(1-\mu) <1/2$ and $|Im(\mu)|>Re(1-\mu)$.

So assume $0<Re(\mu) \leq 1/2$ and $|Im(\mu)|> Re(\mu)$.  Then $|a^*(0)^t {\bf 1}| = t\mu = \lambda$ for $t \in \mathbb{N}$.  $|Im(\lambda)| = |tIm(\mu)|>tRe(\mu) = Re(\lambda)$.  Thus for an infinite number of $\lambda \in Spec_{_\mu M} L^\mu(0)$, $Re(\lambda) \geq |Im(\mu)|$ is not satisfied, and thus $_\mu M$ is not  an $\Omega$-generated $ \CR$-graded vertex operator algebra.

 {\bf Case 5:} If $Re(\mu) >1$ or $Re(\mu)<0$ then $_\mu M$ has nonzero weight spaces $_\mu M_\lambda$ with both $Re(\lambda)$ arbitrarily large negative and  arbitrarily large positive.  This can be seen by considering that $|a(-1)^k {\bf 1}| = k(1-\mu)$ and $|a^*(0)^t {\bf 1}| = t\mu$, for $k,t \in \mathbb{N}$.  This then implies that ${\bf 1} \notin \Omega(_\mu M)$ even though ${\bf 1} \in \, _\mu M_0$, since if $Re(\mu)>1$ then $a(-1){\bf 1} = (a(-1))_{-1} {\bf 1} \neq 0$ with $|a(-1){\bf 1}| = 1-\mu$, but $Re(1-\mu) - (-1) -1 = Re(1-\mu) <0$. Or analogously, if $Re(\mu) <0$ then $a^*(0){\bf 1} = (a^*(0))_{-1} {\bf 1} \neq 0$ with $|a^*(0){\bf 1} = \mu$, but $Re(\mu) - (-1) -1 = Re(\mu) <0$.  Thus in this case $_\mu M_0 \not\subseteq \Omega(_\mu M)$. 
 
In fact, in this case, $\Omega( _\mu M) = 0$.  To see this, without loss of generality assume $Re(\mu) >1$.  Then consider $v=a(-m_1-1)\cdots a(-m_k - 1)a^*(-n_1) \cdots a^*(-n_t) {\bf 1} \in _\mu M$. We have that $v \notin \Omega(_\mu M)$ since for $u = a(-1){\bf 1} \in _\mu M_{1-\mu}$, we have $(a(-1){\bf 1})_{-1} v = a(-1)v \neq 0$ even though $-1 \neq 1-\mu -1 = - \mu$ and $-1 > - Re(\mu) = Re(1-\mu) - 1$.  Extending by linearity, it follows that  $\Omega(_\mu M) = 0$.

Therefore, in this case $_\mu M$ is not  an $\Omega$-generated $\mathbb{C}$-graded vertex algebra.  Thus, in particular, it is also not  an $\Omega$-generated $\CR$-graded vertex operator algebra.  

In summary, we have that,  independently of its conformal structure, the Weyl vertex algebra is always strongly finitely generated (see Remark \ref{bgrem} (2)). However, as shown above, this vertex algebra is  not necessarily $\Omega$-generated (as this does depend on the conformal structure). Thus, we have the following Theorem:
 
\begin{thm}\label{mu-theorem} 
The $\mathbb{C}$-graded Weyl vertex algebra with grading given by the conformal element $\omega_\mu$ for $\mu \in \mathbb{C}$, denoted
$_\mu M = (M, \omega_\mu)$, is a  finitely strongly generated $\mathbb{C}$-graded vertex algebra that is also a conformal $\mathbb{C}$-graded vertex algebra.  Furthermore:\\

\noindent
I. $_\mu M$ is  a finitely $\Omega$-generated $\CR$-graded vertex operator algebra, i.e. is in $\Omega VOA(\CR(\mathcal{V}))$, if and only if one of the following hold:

(i) $0 < Re(\mu) \leq 1/2$ and $|Im(\mu)| \leq Re(\mu)$;

or 

(ii) $0< Re(1-\mu)< 1/2$ and $|Im(\mu)| \leq Re(1-\mu)$.\\

In addition, in this case I, we have that $\Omega(_\mu M) = \mathbb{C} {\bf 1}$.\\

\noindent
II. $_\mu M$ is both  a finitely $\Omega$-generated $\CR$-graded vertex algebra and a conformal $\mathbb{C}$-graded vertex algebra, i.e. is in $Conf(\mathbb{C}(\mathcal{V})) \cap  \Omega (\CR(\mathcal{V}))$, if and only if $0\leq Re(\mu) \leq 1$. \\

\noindent  III. If $Re(\mu)>1$ or $Re(\mu)<0$ then $_\mu M$ is a finitely strongly generated $\C$-graded conformal vertex algebra but it is not an  $\Omega$-generated $\C$-graded vertex algebra. \\

 Lastly, outside of the strip given by $0 \leq Re(\mu) \leq 1$, we have $\Omega(_\mu M) = 0$, and inside $\Omega(_0 M)$ and $\Omega(_1 M)$ are infinite, whereas elsewhere inside this strip $\Omega(_\mu M) = \mathbb{C} {\bf 1}$.
\end{thm}
\begin{proof} The Cases 1--5 above exhaust the possibilities for $\mu \in \mathbb{C}$, and in each case it is shown that $_\mu M$ is both a conformal $\mathbb{C}$-graded vertex algebra and  an $\Omega$-generated $\mathbb{C}$-graded vertex algebra with generating set $\{a(-1){\bf 1}, a^*(0) {\bf 1} \}$.

Cases 2 and 4(a) show that when $0<Re(\mu)\leq 1/2$, and $|Im(\mu)|\leq Re(\mu)$, or when $0< Re(1-\mu) < 1/2$ and $|Im(\mu)| \leq Re(1-\mu)$ then $_\mu M$ is  an $\Omega$-generated  $\CR$-graded vertex operator algebra. 

Cases 1, 3, 4(b), and 5 show that in the remaining cases, $_\mu M$ is not  an $\Omega$-generated   $\CR$-graded vertex operator algebra but is still both a conformal $\mathbb{C}$-graded vertex algebra and  an $\Omega$-generated $\mathbb{C}$-graded vertex algebra.   

Finally, we note that by Case 1-5, we have that if $0\leq Re(\mu) \leq 1$ (Cases 1-4) then the we have that $_\mu M$ is  an $\Omega$-generated $\CR$-graded vertex algebra, but outside of this region (Case 5) it is not.   

The $\Omega(_\mu M)$ are as given in Cases 1-5.
\end{proof}

  \begin{figure}[ht]
    \centering
    \includegraphics[width=9cm]{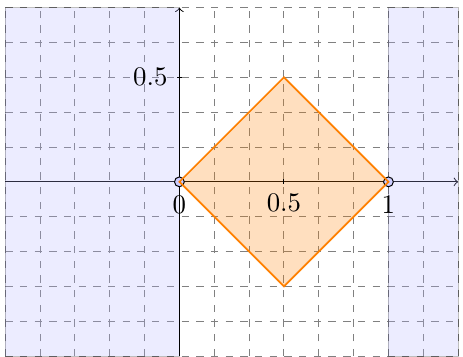}
    \caption{ {\bf Different $\mathbb{C}$-graded vertex algebra structures for $_\mu M$  under conformal flow. For  $\mu \in \mathbb{C}$ satisfying Theorem \ref{mu-theorem}, part I i.e. for $\mu$ inside the diamond shaped region, $_\mu M$  is  an $\Omega$-generated $\CR$-graded vertex operator algebra}.   In the regions where $Re(\mu)> 1$ or $0 < Re(\mu)$, $_\mu M$ is not  an $\Omega$-generated $\CR$-graded vertex algebra  and thus is also not an $\Omega$-generated $\CR$-graded vertex operator algebra. In the remaining regions, $_\mu M$ has the structure of  an $\Omega$-generated $\CR$-graded vertex algebra and a conformal vertex algebra {\it but not} of  an $\Omega$-generated $\CR$-graded vertex operator algebra.}
    \label{fig: valuesformu}
  \end{figure}\hfill

\begin{rmk}
We note that the Weyl vertex algebras in case 3 above, $_\mu M$ with $\mu$ a purely imaginary non zero number, provide a family of examples of conformal $\CR$-graded vertex algebras which are not $\CR$-graded vertex {\it operator} algebras.
\end{rmk}


\begin{rmk} \label{conformalflowrem}
Theorem \ref{mu-theorem} shows a stark contrast between the Weyl vertex algebras (i.e. free bosonic ghosts) under conformal flow in comparison to free bosons under conformal flow.  Recall (cf.  \cite{MN} ) that free bosons $V_{bos} = \mathbb{C}[\alpha(-n) \; | \; n \in \mathbb{Z}_+]$ admit a family of conformal vectors $\omega_\nu = \frac{1}{2} \alpha(-1)^2 {\bf 1} + \nu \alpha(-2) {\bf 1}$,  for $\nu \in \mathbb{C}$  which endow $V_{bos}$ with a vertex operator algebra structure of central charge $1-12\nu^2$.  However under this conformal flow we have that $L^\nu(0)$, and thus the $\mathbb{Z}$-grading and vertex operator algebra structure of $V_{bos}$ does not change.  In fact the vertex operator algebra $(V_{bos}, \omega_\nu)$ differs from $(V_{bos}, \omega_{\nu'})$ only in its representation theory in that some indecomposable nonirreducible modules can have a $L^\nu(0)$-action that is semisimple on a module with a 2   $\times$ 2 Jordan block in the vacuum space $\Omega(W)$ for some values of $\nu$ whereas the action of $L^{\nu'}(0)$ is not semi-simple for some other $\nu'\neq \nu$. But the non semi-simplicity of $V_{bos}$-modules under conformal flow does not change, i.e. $(V_{bos}, \omega_\nu)$ is always irrational.  Whereas by comparison, we will show in Section 5.3 that $(_\mu M, \omega_\mu)$ can be rational or irrational depending on $\mu$.
\end{rmk}

\begin{rmk} We note that the use of the term ``rank'' in the definition of the Weyl vertex algebra in Section 3 alludes to the number of pairs of fields of the form $a(z), a^*(z)$ (or pairs of $\beta(z), \gamma(z)$ fields in the Physics literature) and not the rank of the vertex algebra in the sense of \cite{FLM3} which involves the action of the Virasoro algebra and is instead referred to as the central charge. Therefore, the rank $n$ Weyl vertex algebra, consists of the tensor product of $n$ copies of the rank 1 Weyl vertex algebra. Its conformal structure is determined by the choice of the conformal structures associated to each of the tensor factors. Namely, the rank $n$ Weyl vertex algebra $_{\mu_{1}} M\otimes \dots \otimes  \ _{\mu_{n}}M$ has conformal vector $\omega=\sum_{i=1}^n \omega_{\mu_{i}}$ (as in equation \ref{conformal}) and central charge $c=\sum_{i}^n c_{\mu_{i}}$ (as in equation (\ref{central})).
\end{rmk}

\section{4. Zhu algebras of $\Omega$-generated   $\CR$-graded vertex algebras}

 In this section we present some results from \cite{LM} on the (level zero) Zhu algebras of $\Omega$-generated $\CR$-graded vertex  algebras, and the correspondence between simple modules for the Zhu algebra and simple $\CR$-graded modules for the vertex algebra.  We remind the reader that in \cite{LM}, our notion of $\Omega$-generated $\CR$-graded vertex  algebra was called a $\mathbb{C}$-graded vertex algebra.

Let $V=(V,Y,{\bf 1})$ be  an $\Omega$-generated  $\CR$-graded vertex algebra with grading $V=\bigoplus_{\lambda\in\mathbb{C}}V_{\lambda}$. For $u\in V_{\lambda}$, let $\overline{
|u|}$ denote the ceiling of the real part of $\lambda$, i.e., \[\overline{|u|}:=\min\{n\in\mathbb{Z}~|~n\geq Re(\lambda) = Re(|u|)\}.\] Let 
\begin{equation} \label{vr}
 V^r:= \mathrm{span}_\mathbb{C} \{v\in V~|~  \mbox{$v$  is of homogeneous weight and } r=|v|-\overline{|v|}\}.
\end{equation} 
Then, we have that  $$V=\bigoplus_{r\in \mathbb{C}}V^r.$$ 

\begin{rmk}\label{zgradingv0}
We observe that $u\in V^0$ if and only if $|u|=\overline{|u|}$. Equivalently, if and only if $Lu=\overline{|u|}u$  where $L$ is the operator defined in Eqn.\ (\ref{Ldef}). We can then characterize $V^0$ as the vertex subalgebra of $V$ consisting of all vectors of integer weight.
\end{rmk}

In \cite{Z}  Zhu introduced an associative algebra, $A(V)$ associated to any vertex operator algebra $V$ which can be used to classify its irreducible representations. Laber and Mason in \cite{LM} studied the Zhu algebra associated to certain $\mathbb{C}$-graded vertex algebras by making the necessary modifications to the formulas introduced by Zhu. We will use the $\mathbb{C}$-graded Zhu algebra machinery to show that a particular family of $\Omega$-generated $\CR$-graded vertex algebras are rational. We recall first the appropriate definition of the Zhu algebra in the $\mathbb{C}$-grading setting following\cite{LM}.
\begin{dfn}\cite{LM} Let $V$ be  an $\Omega$-generated  $\CR$-graded vertex algebra.
Let $u\in V^r$  and $v \in V$. Define the products $\circ$ and $*$ on $V$ as the linear extensions of the following:
$$u\circ v:=\Res_z\frac{(1+z)^{\overline{|u|}+\delta_{r,0}-1}}{z^{1+\delta_{r,0}}}Y(u,z)v$$
and
\begin{equation}\label{define-*}
u*v=\delta_{r,0}\Res_z\frac{(1+z)^{|u|}}{z}Y(u,z)v.
\end{equation}
Define $O(V)$ to be the linear span of all elements of the form $u\circ v$ for $u,v\in V$.
\end{dfn}
\begin{prop}\label{Zhu-properties-prop}\cite{LM} Let $V$ be  an $\Omega$-generated  $\CR$-graded vertex algebra and define $V^r$ as in equation (\ref{vr}). Then,

(i) If $r\neq 0$ then $V^r\subseteq O(V)$.

\vspace{0.2cm}

(ii) For $u\in V$ homogeneous, $( D+ L)u\equiv 0\mod O(V),$  where $L$ is the operator defined in Eqn.\ (\ref{Ldef})

\vspace{0.2cm}

(iii) For $u\in V^r$ homogeneous,  $v \in V$, and any $m\geq n\geq 0$, we have 
\[\Res_z\frac{(1+z)^{\overline{|u|}+\delta_{r,0}-1+n}}{z^{1+\delta_{r,0}+m}}  Y(u,z)v \in O(V).\]

\vspace{0.2cm}

(iv) For $u,v\in V$ homogeneous, we have 
\[Y(u,z)v\equiv (1+z)^{-|u|-|v|}Y\left(v,-\frac{z}{1+z}\right)u\mod O(V).\]

\vspace{0.2cm}

(v) For $u,v \in V^0$ homogeneous, we have the identities 
\[u*v \equiv \Res_z\frac{(1+z)^{|v|-1}}{z}Y(v,z)u\mod O(V)\]
   and 
\[u*v-v*u \equiv\Res_z(1+z)^{|u|-1}Y(u,z)v\mod O(V). \]

(vi) $O(V)$ is a 2-sided ideal of $V$ with respect to the $*$ product.

\vspace{0.2cm}

(vii) Define $A(V):=V/O(V)$. Then $A(V)$ is an associative algebra with respect to the $*$ product.

\end{prop}

\begin{rmk} \label{v0zhu}
It follows from Proposition \ref{Zhu-properties-prop} $(i)$ and $(vii)$ that $A(V)=V^0/(O(V)\cap V^0)$
\end{rmk}
\vspace{0.2cm}
 \begin{dfn}
For $v\in V$ homogeneous, define the  {\em zero mode $o(v)$ of }$v$ as $o(v)=v_{|\overline{v}|-1}$. We extend this definition to all of $V$ by linearity.  
\end{dfn} 

\begin{rmk} If $v\in V^r$ for some nonzero $r$ then $\overline{|v|}-1> Re(|v|)-1$, and $o(v)=v_{\overline{|v|}-1}$ annihilates any element of $\Omega(V)$. Also, if $v\in V^0$  then $|\overline{v}|=|v|$, and this definition of $o(v)$ reduces to the original zero mode definition given  by Zhu in \cite{Z}.
\end{rmk}
We conclude this section stating the expected correspondence between simple $A(V)$ modules and simple ``admissible'' $V$-modules in the $\mathbb{C}$-graded setting. We note that $\CR$-graded modules are the appropriate ``admissible" representations in this context.
\begin{prop} \label{correspondence}\cite{LM} Let $V$ be  an $\Omega$-generated $ \CR$-graded vertex algebra.\\
(i) Let $W=W(0)\bigoplus_{\lambda \in\mathbb{C}, Re(\lambda)>0}W(\lambda)$  be a simple   $\mathbb{C}_{Re>0}$-graded $V$-module with  $\Omega(W)=W(0)$, as shown in Proposition \ref{strongly-generated-LM-prop}.  Then, $\Omega(W)$ is a simple $A(V)$-module. \\
(ii) There is a one-to-one correspondence between the categories of simple $A(V)$-modules and simple    $\mathbb{C}_{Re>0}$-graded $V$-modules.
\end{prop}

\section{5. Rationality for certain $\CR$-graded vertex  operator algebras and applications}\label{sec3} \
 In this section we prove our main result on the rationality of finitely  $\Omega$-generated  $\CR$-graded vertex  operator algebras that are not $\mathbb{Z}$-graded and whose simple $\CR$-graded modules are all ordinary. We then apply this result to the Weyl vertex algebras with the central charges $c_\mu$ (or equivalently the conformal element $\omega_\mu$) that give $(_\mu M, \omega_\mu)$ the structure of a $\CR$-graded vertex operator algebra.

 The following Theorem is analogous to Theorem 3.3 in \cite{DN}  where a $g$-rationality for $g$-twisted modules of a vertex operator algebra $V$ and for an automorphism $g$  was studied.  Here we use the idea of their proof applied to the setting of   $\Omega$-generated $\CR$-graded vertex operator algebras. 
 
\begin{thm}\label{rationalcondition} 
Let $V$ be  an $\Omega$-generated $ \CR$-graded vertex operator algebra satisfying  the following conditions:

\begin{enumerate} 
\item Every simple $\CR$-graded $V$-module is an ordinary module.
\item $A(V)$ is a finite-dimensional semisimple associative algebra.
\item $\omega+O(V)$ acts via its zero mode $L(0)$ on all irreducible $A(V)$-modules as the same constant eigenvalue $\lambda$---that is, there is a fixed $\lambda \in \mathbb{C}$ such that for any $A(V)$-module $U$, then  $U$ consists of generalized eigenvectors for $o(\omega) = L(0)$ with eigenvalue $\lambda$.
\end{enumerate}

Then $V$ is rational, i.e., every $\mathbb{C}_{Re>0}$-graded $V$-module is completely reducible.
\end{thm}

\begin{proof} Let $W$ be a $\mathbb{C}_{Re>0}$-graded $V$-module. We will show that $W$ is a completely reducible $\CR$-module by considering the following cases:
\begin{enumerate}
\item[Case 1]  {\bf $\Omega(W)$ is a simple $A(V)$-module and $W$ is generated by $\Omega(W)$.}  If $\tilde W$  is a $V$-submodule of $W$, then  $\Omega(\tilde W)$ is an $A(V)$-submodule of $\Omega(W)$ and therefore by the assumption that $\Omega(W)$ is simple, $\Omega(\tilde W)$ must be the trivial $A(V)$-module or $\Omega(W)$.  Since $W$  (and thus $\tilde W$) is generated by $\Omega(W)$, this implies that $\tilde W = 0$ or $\tilde W$ is generated by $\Omega(W)$ and is thus $W$, assuming $W \neq 0$. Therefore $W$ is a simple $\mathbb{C}_{Re>0}$-graded $V$-module.

\item[Case 2] { \bf $\Omega(W)$ is not a simple $A(V)$-module and $W$ is generated by $\Omega(W)$.}  Since $\Omega(W)$ is not simple and $A(V)$ is  finite-dimensional semisimple, $\Omega(W)$ is the direct sum of simple $A(V)$-modules, say $\Omega(W) = \bigoplus_{i \in I} \Omega(W)^i$ where $\Omega(W)^i$ is simple for $i \in I$, for $I$  some indexing set.  Thus, if $W$ is generated by $\Omega(W)$, F we set $W^i:=\langle \Omega(W)^i\rangle$ and we obtain $W = \bigoplus_{i \in I} W^i$ where each $W^i$ is  generated by the simple $A(V)$-module $\Omega(W)^i$ and is thus simple by the  Case 1 argument.  Therefore, when $W$ is generated by its  lowest weight vectors,  in this case $\Omega(W)$, we have that $W$ is a completely reducible $\mathbb{C}_{Re>0}$-graded $V$-module. 

\item[Case 3]  {\bf $W$ is not generated by $\Omega(W)$}. We will show that $W$ is completely reducible by further analyzing two subcases. First we
let $\tilde W$ be the submodule of $W$ generated by $\Omega(W)$.  Then $\tilde W$ is completely reducible by Case 2, i.e. $\tilde W = \bigoplus_{i \in I} \tilde W^i$ for $\tilde W^i$ irreducible.  Thus $\tilde W(0) = \bigoplus_{i \in I} \tilde W^i(0) = \bigoplus_{i \in I} \Omega(\tilde W^i) = \Omega(W)$ by Proposition \ref{strongly-generated-LM-prop}.  Then  we have $\tilde W= \tilde W (0) \bigoplus_{\mu\in\mathbb{C}, Re(\mu)>0} \tilde W(\mu)$ with $\tilde W(0)$ a generalized eigenspace for $L(0) = \omega^{\tilde W}_1$ with eigenvalue $\lambda$, so that $\mathrm{Spec}_{\tilde W(0)} L(0) = \lambda$. Moreover, we have $\tilde W(0) = \Omega(W)$ generates $\tilde W$, and for $m\in \mathbb{Z}$ and $v \in V$, we have that $L(0)$ acts on $v_m \tilde W(0)$ via the $L(0)$-eigenvalue $\mathrm{wt} \, v -m -1 + \mathrm{wt} \,  w$  for $w \in \tilde W(0)$. Thus $\tilde W$ is graded by $L(0)$-generalized eigenspaces with eigenvalues of the form  $(\mathrm{Spec}_V L(0) + \lambda + \mathbb{N} ) \cap \{ \mu + \lambda \in \mathbb{C} \; | \; \mu \in \mathrm{Spec}_VL(0) \}  = \mathrm {Spec}_V L(0) + \lambda$, i.e.
\[\tilde W  = \tilde W_{\lambda}  \bigoplus_{\mu \in \mathbb{C}, \, Re(\mu) > 0} \tilde W_{\lambda + \mu}  = \tilde W_\lambda  \bigoplus_{{\mu \in \mathrm{Spec}_V L(0)}\atop{\mu\neq 0}} \tilde W_{\lambda + \mu} ,\]
with $\tilde W_\lambda =\tilde  W(0)$ (where we have used the fact that $v_m \tilde W(0) = 0$ for $m>0$ since $\tilde W(0) = \Omega(W)$).
Next, we consider the module $W/\tilde W$.  Since we are assuming $W$ is not generated by $\Omega(W)$, we have that $W/\tilde W \neq 0$. This implies $\Omega (W/\tilde W)  \neq 0$.  We analize the following two subcases to show that under the assumptions of Case 3, $W$ must be completely reducible:

Case 3I. Suppose $W/\tilde W$ is completely reducible.  Then as above, $(W/\tilde W)(0) = \Omega(W/\tilde W)$ and every $w + \tilde W \in W/\tilde W$ is contained in some $L(0)$-generalized eigenspace, i.e.  $w + \tilde W \in (W/\tilde W)_{\lambda + \mu}$ for some $\mu \in \mathrm{Spec}_V L(0)$.  But then $W$ itself is $(\mathrm{Spec}_VL(0) + \lambda)$-graded by $L(0)$-generalized eigenspaces.  In addition $W_\lambda$ is an $A(V)$-module and thus completely reducible.  Thus the submodule of $W$ generated by $W_\lambda$ is completely reducible. Denote this by $W'$.   But then if $W \neq W'$ there exist elements in $\Omega(W/W')$ that are not in $ W_\lambda$ and thus are in a generalized eigenspace for $L(0)$ of the form  $\lambda + \mu$ for $\mu \neq 0$ which contradicts the fact that $\Omega(W/W')$ is an $A(V)$-module and thus in the $\lambda$ generalized eigenspace for $L(0)$.  Therefore $W = W'$ and $W$ is completely reducible. 

Case 3II.  Suppose $W/\tilde W$ is not completely reducible.  Then replace $W/\tilde W$ with the submodule of $W/\tilde W$ generated by $\Omega(W/\tilde W)$, which is $U/\tilde W$ for some submodule $U$ of $W$.  Then by the argument above, since $U/\tilde W$ is completely reducible, every $u + \tilde W \in U/\tilde W$ is contained in some $L(0)$-generalized eigenspace, i.e.  $u + \tilde W \in (U/\tilde W)_{\lambda +\mu}$ for some $\mu \in \mathrm{Spec}_VL(0) \smallsetminus 0$.  But then $U$ itself is  $(\mathrm{Spec}_VL(0) + \lambda)$-graded by $L(0)$-generalized eigenspaces.  In addition $U_\lambda$ is an $A(V)$-module and thus completely reducible.  Thus the submodule of $W$ generated by $U_\lambda$ is completely reducible. Denote this by $U'$.   But then if $U \neq U'$ there exist elements in $\Omega(U/U')$ that are not in $U_\lambda$ and thus are in a generalized eigenspace for $L(0)$ of the form  $\lambda + \mu$ for $\mu \neq 0$ which contradicts the fact that $\Omega(U/U')$ is an $A(V)$-module and thus in the $\lambda$ generalized eigenspace for $L(0)$.  Therefore $U = U'$ and is completely reducible.  

Finally, we will show below that $\tilde W$, the submodule of $W$ generated by $\Omega(W)$, is a maximal completely reducible submodule of $W$.  This would then imply that $\tilde W \subset U \subset \tilde W$, implying $\Omega(W/\tilde W)$ generates the trivial module $\tilde W/\tilde W$, and thus must be the trivial $A(V)$-module, which implies $W = \tilde W$, and so $W$ is completely reducible. 

Therefore we only have left to show that $\tilde W$ is a maximal completely reducible submodule of $W$.  Indeed, if $U$ is also a maximal completely reducible submodule, then both $\tilde W$ and $U$  are generated by $\Omega(W)$, and thus equal.  
\end{enumerate}
This completes the proof. 
\end{proof}

\begin{rmk} Note that in the proof above, one of the key facts used repeatedly is that for the class of $\mathbb{C}$-graded vertex algebras that we are working dealing with, namely  $\Omega$-generated $\CR$-graded vertex operator algebras, we have that $\Omega(W) = W(0)$ for all simple $\CR$-graded $V$-modules $W$, i.e. Proposition \ref{strongly-generated-LM-prop} (ii) holds.
\end{rmk}

\subsection{5.1 Filtration of Zhu algebras}
 
In \cite{Z} Zhu introduced two associative algebras related to a vertex operator algebra $V$, the Zhu algebra $A(V)$ and the $C_2$ algebra $V/C_2(V)$. Moreover, to prove that $V/C_2(V)$ is a Poisson algebra, Zhu built a filtration further studied and generalized by Li in \cite{Li5}. In this section, we use analogous constructions to describe $A(V)$ in the $\CR$-grading setting.

Let $V$ be  an $\Omega$-generated  $\CR$-graded vertex algebra with grading $V=\bigoplus_{\mu\in\mathbb{C}}V_{\mu}.$   We continue with the notation from the last section, and in particular of $\overline{|u|}$ for the ceiling of the real part of $\mu$ if  $u\in V_{\mu}$, and $V^r$ for the set of all elements $u\in V$ with $r=|u|-\overline{|u|}$.  And recall that then $V=\bigoplus_{r\in\mathbb{C}}V^r$ and for $r\neq 0$, we have $V^r\subseteq O(V)$. Observe that 
\begin{equation}\label{Zhu-only-zero}
A(V)=V/O(V)=V^0+O(V)/O(V)=\bigoplus_{n\in \mathbb{Z}}V_n+O(V)/O(V).
\end{equation} by Remarks \ref{zgradingv0} and \ref{v0zhu}.

Now, we further assume that the integer grading for $V^0$ is bounded below. For the purposes of the exposition, we write  $V^0=\bigoplus_{ n\in \mathbb{N}} V_n$  although the results will follow with little modification if $V^0 = \bigoplus_{n = N}^\infty V_n$ for some $N \in\mathbb{Z}$. Consider the filtration $\{F_t A(V)\}_{t\in \mathbb{N}}$ where \[F_t A(V)=\bigoplus_{j=0}^tV_j+O(V)/O(V)  \subset A(V).\] 
From the definition of $*$, i.e. Eqn.\ (\ref{define-*}), and by taking the residue, we have that for $u\in V^r,v\in V$, 
\begin{equation}\label{filtration1}
u*v \ = \ \delta_{r,0}\Res_z\frac{(1+z)^{|u|}}{z}Y(u,z)v \ = \ \delta_{r,0}\sum_{i\geq 0}{|u|\choose i}u_{i-1}v. 
\end{equation}  
 Also note that for homogeneous elements $u,v\in V$, we have $|u_{i-1}v|=|u|+|v|-i$.
Hence, we can conclude that 
\begin{equation} \label{multiply-grading}
F_s A(V) * F_t A(V)\subseteq F_{s+t}A(V) \qquad \mbox{ for $s,t\in\mathbb{N}$.} 
\end{equation}

 Letting $F_{-1}A(V)=0$, we define an  $\mathbb{N}$-grading on $A(V)$ by
\[gr_tA(V)=F_tA(V)/F_{t-1}A(V)\text{ for }t\in\mathbb{N},\]
so that 
\[grA(V)=\bigoplus_{t=0}^{\infty}gr_t A(V).\]

Observe that  by Eqn.\ (\ref{multiply-grading}), the multiplication of $A(V)$ induces an associative multiplication on the graded vector space $grA(V)$. In addition, we have the following lemma.

\begin{lem} Let $V$ be an  $\Omega$-generated $\CR$-graded vertex algebra such that $V^0=\bigoplus_{n=0}^{\infty}V_n$. Then $grA(V)$ is a commutative associative algebra. 
\end{lem}

\begin{proof} Let $u\in V^r$, $v\in V^s$. By the fact that $V^r \subset O(V)$ if $r \neq 0$, i.e. Proposition  \ref{Zhu-properties-prop}$(i)$  if either $r\neq0$ or  $s\neq 0$ then $u*v=0=v*u$, since either $u$ or $v$ is in $O(V)$. So assume that $u,v\in V^0$ are homogeneous elements, and $u\in F_tA(V)$, $v\in F_sA(V)$. Since by Proposition \ref{Zhu-properties-prop}$(v)$ 
\begin{eqnarray*}
u*v-v*u &=& \Res_z(1+z)^{|u|-1}Y(u,z)v\mod O(V)\\
&=& \Res_z\sum_{i\geq 0} {|u|-1\choose i}z^iY(u,z)v\mod O(V)\\
&=& \Res_z\sum_{i\geq 0}{|u|-1\choose i}u_iv\mod O(V)
\end{eqnarray*}
and $| u_iv|=|u|+|v|-i-1$, we can conclude that, modulo $O(V)$, $u*v-v*u\in F_{s+t-1}A(V)$, i.e., is zero in $gr_{t+s}A(V)$. Hence, $gr(V)$ is commutative.
\end{proof}
\begin{rmk} \label{dim} $grA(V)$ is isomorphic to $A(V)$ as a vector space.
\end{rmk}

Next, we study an upper bound for $\dim A(V)$ when $V$ is an $\Omega$-generated $\CR$-graded vertex algebra. For simplicity, we write $[u]$ for $u+O(V)$.  Consider the linear epimorphism 
\begin{eqnarray}\label{define-f}
f:V &\longrightarrow& grA(V)\\  
u & \mapsto & [u]+F_{k-1}A(V), \qquad \mbox{ for $u\in V_k$. } \nonumber 
\end{eqnarray}
Note that if $u\in V^r$ and $r>0$ then $f(u)=[0]+F_{k-1}A(V)$. Consequently, $u\in Ker(f)$.

 Now, let $u,v$ be homogeneous elements in $V$. 

\vspace{0.1cm}

\noindent\underline{Case 1:} $u\in V^0$. 

\noindent Since $u\circ v=\Res_z\frac{(1+z)^{\overline{|u|}}}{z^2}Y(u,z)v=\sum_{i\geq 0}{\overline{|u|}\choose i}u_{i-2}v\in O(V)$ and $$|u_{j-2}v|=|u|+|v|-j+1\leq |u|+|v| \text{ when }j\geq 1,$$ we have $$f(u_{-2}v)=[u_{-2}v]+F_{|u|+|v|}A(V)=[0]+F_{|u|+|v|}A(V).$$ 
Moreover, $u_{-2}v\in Ker(f)$.

\vspace{0.1cm}

\noindent\underline{Case 2:} $u\in V^r$ such that $r>0$. 

\noindent Since $u\circ v=\Res_z\frac{(1+z)^{\overline{|u|}-1}}{z}Y(u,z)v=\sum_{j\geq 0}{\overline{|u|}-1\choose j}u_{j-1}v\in O(V)$ and 
$$|u_{j-1}v|=|u|+|v|-j<|u|+|v|\text{ for all }j\geq 1, $$ these imply that $$f(u_{-1}v)=[u_{-1}v]+F_{|u|+|v|-1}A(V)=[0]+F_{|u|+|v|-1}A(V).$$ In addition, $u_{-1}v\in Ker(f)$.

\noindent In conclusion, we have the following theorem.

\begin{thm}\label{upperbound} Let $V$ be an $\Omega$-generated $\CR$-graded vertex algebra with  integer graded part $V^0$ as in equation \ref{vr} satisfying $V^0=\bigoplus_{n=0}^{\infty}V_n$. Let $f:V\rightarrow gr A(V)$ be defined by (\ref{define-f}). Set 
\[ C(V)  =Span_{\mathbb{C}}\{a,u_{-2}v, b_{-1}w~|~ a,b \in V^r, \mbox{with $r\neq 0$},  ~u\in V^0,~ \mbox{and}~ v,w\in V\}.\]
Then $ C(V)  \subseteq Ker(f)$. In addition $f$ induces a linear epimorphism $\bar{f}$ from  $V/C(V)$ to $grA(V)$ and therefore,  $\dim A(V)\leq \dim  V/C(V)$. 
\end{thm}

 Recall that an $\Omega$-generated $\CR$-graded vertex algebra $V$ is endowed with the endomorphism $D$ as defined in Remark \ref{D}. Using the fact that $ (Du)_n=-nu_{n-1}$ for $u\in V$, $n\in\mathbb{Z}$, we  have the following corollary:

\begin{cor}\label{elementsW}
 Let $V$ be a  $ \CR$-graded vertex algebra such that $V^0=\bigoplus_{n=0}^{\infty}V_n$. Then
\begin{enumerate}
\item for $u\in V^0$, $v\in V$,  we have $u_{-n}v\in  C(V)$ for all $n\geq 2$;
\item for $b\in V^r$ with $r\neq 0$, and $w\in V$, then $b_{-m}w\in  C(V)$ for all $m\geq 1$.
\end{enumerate}
\end{cor}

\begin{rmk} It is necessary to assume that $V^0=\bigoplus_{n\in\mathbb{Z}}V_n$ is bounded below. Otherwise the theorem above is false. For instance, when $L$ is a nondegenerate non positive definite even lattice of an arbitrary rank, it was shown in \cite{Y} that $A(V_L^+)\neq 0$ and in \cite{JY} that $\dim (V_L^+/C(V_L^+))=0$. In that context, $C(V_L^+)=C_2(V_L^+)$, the $C_2$ space originally defined by Zhu \cite{Z} for $\mathbb{Z}$-graded vertex algebras. However, as noted earlier, it is enough to just assume the grading for $V^0$ is bounded from below, not necessarily by zero.
\end{rmk}

\subsection{5.2 Main results on rationality for certain $\CR$-graded vertex  operator algebras} 
 In this section, we use the construction of the Zhu algebra $A(V)$ presented above to prove that under some mild conditions an $\Omega$-generated $\CR$-graded vertex algebra admits only one $\CR$-graded simple module.

\begin{thm}\label{A(V)} Let $V$ be  a finitely $\Omega$-generated  $\CR$-graded vertex algebra generated as in Remark \ref{strogen} by $u^1,\dots ,u^k$.  Let $V^0$ be the integer graded part of $V$ as in Eqn.\ \ref{vr}. Assume that 
\begin{enumerate}
\item For $j\in \{1,\dots ,k\}$, $|u^j|$ is not an integer. Namely,the strong generators satisfy $u^j \in V \setminus V^0$ for $1\leq j\leq k$ ;
\item $V^0=\bigoplus_{n=0}^{\infty}V_n$. 
\end{enumerate}
Then, $\dim A(V)=1$ and $A(V)\cong \mathbb{C}$. 
\end{thm}

\begin{proof} By  Corollary \ref{elementsW}, using that $u^j\in V^r$ for $r\neq 0$ we can conclude that $u^{j_1}_{-n_1}\cdots u^{j_t}_{-n_t}{\bf 1}\in  C(V)$ for all $u^{j_i}\in\{u^1,\dots ,u^k\}$, $n_i \geq 0$. Hence, $V/ C(V) =\mathbb{C}{\bf 1}+  C(V)$. Moreover, in light of Theorem \ref{upperbound} this implies that, $\dim A(V)=1$ and $A(V)\cong \mathbb{C}$ as desired.
\end{proof}

\begin{thm}\label{Vrational} Let $V$ be  a finitely $\Omega$-generated $\CR$-graded vertex operator algebra that is finitely generated by $u^1, \dots ,u^k$,  as in Remark \ref{strogen} and in addition satisfies the following:
\begin{enumerate}
\item For each $j\in \{1, \dots ,k\}$, $|u^j|$ is not an integer;
\item $V^0=\bigoplus_{n=0}^{\infty}V_n$;  and
\item  Every simple $\CR$-graded $V$-module is ordinary.
\end{enumerate}
Then $V$ is rational, and has only one simple $\mathbb{C}_{Re>0}$-graded $V$-module.
\end{thm}
\begin{proof} Since $V$ is  an $\Omega$-generated  $\CR$-graded vertex operator algebra, using Theorem \ref{A(V)} we have that $\dim A(V)=1$. By Proposition \ref{correspondence} $(ii)$ we can conclude that $V$ has only one simple $\mathbb{C}_{Re>0}$-graded $V$-module. Finally, by Theorem \ref{rationalcondition} and condition (3), we can conclude that $V$ is rational.
\end{proof}

\subsection{5.3 On Rationality of   $\CR$-graded Weyl vertex operator algebras }\label{Weyl-rational-subsection}
 In this section we apply Theorems \ref{rationalcondition} and \ref{A(V)} to Weyl vertex algebras $_\mu M$ with certain conformal structures as classified in Theorem \ref{mu-theorem} to prove the rationality of $_\mu M$, for those values of $\mu$ that give $_\mu M$ the structure of  an $\Omega$-generated $\CR$-graded vertex operator algebra, including, for instance when  $0 < Re(\mu)  < 1$ and $Im(\mu) = 0$, which corresponds to the case of the central charge $c$ real and in the range $-1<c<2$.

\begin{thm}\label{rank1}
Let  $\mu\in\mathbb{C}$ such that one of the following holds:

(i) $0 < Re(\mu) \leq 1/2$ and $|Im(\mu)| \leq Re(\mu)$;

or 

(ii) $0< Re(1-\mu)< 1/2$ and $|Im(\mu)| \leq Re(1-\mu)$.\\ 
Then $(\, _\mu M,\omega_{\mu})$ is a rational   $\Omega$-generated $\CR$-graded vertex operator algebra, and  has only one simple $\mathbb{C}_{Re>0}$-graded module which is in fact a simple ordinary $_\mu M$-module, namely $_\mu M$ itself.
\end{thm}

\begin{proof} 
Theorem \ref{mu-theorem} implies that $_\mu M$ is  an $\Omega$-generated $\CR$-graded vertex operator algebra.  Theorem \ref{A(V)} and the fact that $|a(-1){\bf 1}| = 1-\mu$ and $|a^*(0){\bf 1}| = \mu$ imply that $A(_\mu M) \cong \mathbb{C}$, and thus $_\mu M$ has only one irreducible $\CR$-graded module.  Since $_\mu M$ is a $\CR$-graded irreducible module over itself, we conclude that the only simple $\CR$-module is $_\mu M$.  We observe that in fact $_\mu M$ is an ordinary $_\mu M$-module as well.   Thus by Theorem \ref{rationalcondition},  we have that $_\mu M$ is rational. 
\end{proof}

\begin{cor}\label{rankn} 
For $i\in\{1,...,n\}$,  if $\mu_i \in \mathbb{C}$ and one of the following holds for each $\mu_i$:

(i) $0 < Re(\mu_i) \leq 1/2$ and $|Im(\mu_i)| \leq Re(\mu_i)$;

or 

(ii) $0< Re(1-\mu_i)< 1/2$ and $|Im(\mu_i)| \leq Re(1-\mu_i)$. \\
Then, $(_{\mu_1} M,\omega_{\mu_1})\otimes......\otimes(\, _{\mu_n} M,\omega_{\mu_n})$ is rational.
\end{cor}
\begin{proof} This follows immediately from Theorem \ref{rank1}.
\end{proof}

 For more general values of $\mu$, namely in the range $0\leq Re(\mu) \leq 1$,  but not necessarily in the subregion defined by $(i)$ and $(ii)$ in Theorem \ref{rank1} and Corolllary \ref{rankn}, we do not necessarily obtain a $\CR$ graded vertex {\it operator} algebra structure in $_\mu M$ but we still have that $_\mu M$ is a finitely $\Omega$-generated $\CR$-graded vertex algebra (see Theorem \ref{mu-theorem} Case II). We conclude this section by showing that these families of Weyl vertex algebras admit only one irreducible $\CR$-graded simple module. 

\begin{thm} \label{ZhuCWeyl}
Let $\mu \in \mathbb{C}\setminus \{0, 1\}$ be such that $0\leq Re(\mu) \leq 1$. Then, the Weyl vertex algebra $_\mu M$ admits a unique, up to isomorphism, irreducible $\CR$-graded module which is $_\mu M$ itself.
\end{thm}
\begin{proof}
By Theorem \ref{mu-theorem} (II), for $\mu \in \mathbb{C}$ such that $0\leq Re(\mu) \leq 1$, the Weyl vertex algebra $_\mu M$ is a finitely $\Omega$-generated $\CR$-graded vertex algebra. Moreover, because $\mu \neq 0,1$ we have from equations (\ref{deggen1}) and  (\ref{deggen2}) that the strong generators of $_\mu M$ have non-integer degree so that condition (1) of Theorem \ref{A(V)} is satisfied. In addition, since $\Omega(_\mu M)=\mathbb{C}\vac$ it is clear that condition (2) of the Theorem also holds. Therefore, we obtain that $A(_\mu M)\cong \mathbb{C}.$ Finally, Proposition \ref{correspondence} $(2)$ implies that $_\mu M$ admits only one irreducible $\CR$-graded module which must be $_\mu M$ itself.

\end{proof}
\begin{rmk}
\begin{enumerate}
    \item We note that in light of Theorem \ref{ZhuCWeyl} we obtain a family of conformal vertex algebras in which the Zhu algebra is one dimensional. In particular, the class of the conformal vector, $[\omega]\in A(V)$ must be a multiple of the class of the vacuum vector $[\vac]$ as in the classical setting of vertex operator algebras constructed from self dual lattices \cite{FLM1}.
    \item  The Weyl vertex algebras admit many non-isomorphic irreducible {\it weak} modules such as the relaxed highest weight modules studied in \cite{RW}. We note, however, that those modules are independent of the conformal structure on the Weyl vertex algebra and are not $\CR$-modules, because they have infinite dimensional graded components. In particular, they are not ``admissible'' modules, namely, modules induced from the level zero Zhu algebra, and thus, possessing a $\CR$-grading.
    Another such example of (non-admissible) non-isomorphic weak modules are the (generalized) Whittaker modules, for which the reducibility was studied in \cite{ AP2, ALPY}.
  
\end{enumerate}

\end{rmk}

\section{6. Summary of applications and future work}
 In this work, we classified the $\mathbb{C}$-graded conformal structures associated to the Weyl vertex algebra. Moreover, we showed that a large family of this vertex algebras admit a unique irreducible ``admissible'' module in the appropriate sense. We also described in detail which families of Weyl vertex algebras admit the $\mathbb{C}$-graded notion of a vertex operator algebra and proved that non-integer $\mathbb{C}$-graded Weyl vertex operator algebras are rational.
In the literature, the Weyl vertex algebra at central charge 2 has been studied in detail (see for instance \cite{L,AW, RW, AP}). This vertex algebra, $_0 M$ in our notation, is not a vertex operator algebra because its graded components fail to be finite dimensional. Linshaw showed in \cite{L} that the (level zero) Zhu algebra $A(_0 M)$ is isomorphic to the rank one Weyl algebra $\mathcal{A}_1$. Higher level Zhu algebras introduced by Dong, Li, and Mason in  \cite{DLM}, can be used to study indecomposable nonirreducible modules.  
Using the theory and methods developed by the first  named author of the current paper, along with Vander Werf and Yang, in \cite{BVY1,BVY2}, and with Addabbo in \cite{AB1, AB2}, preliminary calculations by Addabbo together with the authors of the current paper indicate that the level one Zhu algebra for this Weyl vertex algebra satisfies

 $$A_1(_0 M)\cong \mathcal{A}_1\oplus (\mathcal{A}_1  \otimes Mat_2(\mathbb{C}) ),$$ with $\mathcal{A}_1$ the rank one Weyl algebra.  In particular, the injective image of the level zero Zhu algebra $\mathcal{A}_1$ inside the level one Zhu algebra has a direct sum complement, namely $\mathcal{A}_1 \otimes Mat_2(\mathbb{C})$, and this complement is Morita equivalent to the level zero Zhu algebra $\mathcal{A}_1$. 
 Therefore, there are no new $\mathbb{N}$-gradable $_0 M$-modules detected by the level one Zhu algebra for $_0 M$ that were not already detected by the level zero Zhu algebra.  Thus, this agrees with the work of \cite{AW} on category $\mathcal{F}$ as discussed in the introduction.  Although this shows that the structure of the level one Zhu algebra gives no new information for the admissible $_0 M$-modules, 
we expect that the study of higher level Zhu algebras for $_0 M$ and in the more general $\mathbb{C}$-graded setting will shed light on the difficult open problem of describing the Zhu algebra for an orbifold vertex algebra in which twisted modules are expected to be detected.

\begin{acknowledgments}
We thank the organizers of the Women in Mathematical Physics (WOMAP) conference, Ana Ros Camacho and Nezhla Aghaei where this work started. We are also grateful to the Banff International research station for their (online) hospitality. K. Batistelli thanks Fondecyt for its support. F. Orosz Hunziker thanks the National Science Foundation for its support. V. Pedić Tomić thanks QuantiXLie Center of Excellence for its support. G. Yamskulna thanks the College of Arts and Sciences, Illinois State University for its support. The authors thank Dražen Adamović for insightful discussions on topics related to Weyl vertex algebras and their representation theory. The authors also thank Darlayne Addabbo for her contribution to the preliminary computations on the level one Zhu algebra for the Weyl vertex algebra of central charge 2. The authors are grateful to the referee for their comments and suggestions.

K. Batistelli is supported by Fondecyt project 3190144. F. Orosz Hunziker is supported by the National Science Foundation under grant No. DMS-2102786. V. Pedić Tomić is partially supported by the QuantiXLie Centre of Excellence, a project
cofinanced by the Croatian Government and European Union through the European Regional Development Fund - the Competitiveness and Cohesion Operational Programme
(KK.01.1.1.01.0004). G. Yamskulna is supported by the College of Arts and Sciences, Illinois State University.
\end{acknowledgments}
\appendix
\section{APPENDIX}
\setcounter{equation}{0}
\renewcommand{\theequation}{A.\arabic{equation}}

\subsection{A. Proof of Lemma \ref{grading-lemma}} \label{VCRmod}

 We will prove that if $V$ is  an $\Omega$-generated $\CR$-graded vertex algebra then it satisfies the following:
 
 For $r\geq 1$,   $v^1$,\dots,$v^r$ homogeneous elements in $V$, $n_1$,\dots,$n_r$ integers and $u^0$ a vector in $\Omega(V)$ such that $$v^r_{n_r}v^{r-1}_{n_{r-1}}\cdots v^1_{n_1}u^0\neq 0,$$ then either $$\sum_{j=1}^r(|v^j|-n_j-1 )=0 \qquad \text{ or } \qquad Re\left(\sum_{j=1}^r(|v^j|-n_j-1)\right)>0.$$ 
 
\begin{proof}
We prove the proposition by induction on $r$.
If $r=1$ and $v^1_{n_1}u^0\neq 0$ then because $u^0 \in \Omega(V)$, we have that either $n_1=|v^1|-1$ or $n< Re(|v^1|-1)$. Equivalently, either $|v^1|-n_1-1=0$ or $Re(|v^1|-1-n)>0$, so the proposition holds for $r=1.$

\noindent Next, assume that $r\geq 2$ and that $$v^r_{n_r}v^{r-1}_{n_{r-1}}\cdots v^1_{n_1}u^0\neq 0.$$ Using the inductive hypothesis on $$v^{r-1}_{n_{r-1}} \cdots v^1_{n_1}u^0$$ we know that either
$$\sum_{j=1}^{r-1}(|v^j|-n_j-1) =0 \qquad \text{ or } \qquad Re\left(\sum_{j=1}^{r-1}(|v^j|-n_j-1)\right)>0.$$
We consider the following two cases:
\begin{enumerate}
    \item[Case 1] If either $|v^r|-n_r-1=0$ or $Re(|v^r|-n_r-1)>0$, we can conclude immediately that either $\sum_{j=1}^r (|v^j|-n_j-1)=0$ or $ Re\left(\sum_{j=1}^r(|v^j|-n_j-1)\right)>0$ and we are done with this case.
    \vskip1cm
    \item[Case 2] If $|v^r|-n_r-1\neq 0$ and $Re(|v^r|-n_r-1)\leq 0$.  Before presenting the proof of the Lemma in this case we recall the commutator formula (cf.  Eqn.\ (3.1.9) in \cite{LL}) which holds for $n,m\in \mathbb{Z}$ and any two elements $v, v' $ in a vertex algebra:
    \begin{align} \label{comm}
    [v_n,v'_m]=\sum_{i\geq 0}\binom{n}{i}(v_{i}v')_{m+n-i}    
    \end{align}
   Using (\ref{comm}) we can rewrite
\begin{align} \label{r}
v^r_{n_r}v^{r-1}_{n_{r-1}}v^{r-2}_{n_{r-2}}& \cdots v^1_{n_1}u^0=\\
&v^{r-1}_{n_{r-1}}v^r_{n_r}v^{r-2}_{n_{r-2}} \cdots v^1_{n_1}u^0+\sum_{i\geq 0}{n_r\choose i}(v^r_i v^{r-1})_{n_r+n_{r-1}-i}v^{r-2}_{n_{r-2}} \cdots v^1_{n_1}u^0. \nonumber
\end{align}
We further analyze the following two subcases:

Case 2.I There exists $i\geq 0$ such that $(v^r_i v^{r-1})_{n_r+n_{r-1}-i}v^{r-2}_{n_{r-2}} \cdots v^1_{n_1}u^0\neq 0.$ By the inductive hypothesis we have that 
\begin{align*}
    &|v^r_iv^{r-1}|-n_r -n_{r-1}+i-1+\sum_{j=1}^{r-2}(|v^j|-n_j-1) =0 \textrm{  or  }\nonumber \\  
    &Re\left(|v^r_iv^{r-1}|-n_r -n_{r-1}+i-1+\sum_{j=1}^{r-2}(|v^j|-n_j-1) \right)>0.
\end{align*}

Using Remark \ref{weight-gradingin-remark} (2) we have that 
$|v^r_iv^{r-1}|=|v^r|+|v^{r-1}|-i-1$ so we can conclude that either 
$\sum_{j=1}^r (|v^j|-n_j-1)=0$ or $ Re\left(\sum_{j=1}^r(|v^j|-n_j-1)\right)>0$  and the Lemma holds in Case 2.I.
\vskip.5cm
Case 2.II For all $i\geq 0,$  $(v^r_i v^{r-1})_{n_r+n_{r-1}-i}v^{r-2}_{n_{r-2}}\cdots v^1_{n_1}u^0=0$. Then by (\ref{r}) we have that 
\begin{align*}
v^r_{n_r}v^{r-1}_{n_{r-1}}v^{r-2}_{n_{r-2}} \cdots v^1_{n_1}u^0=v^{r-1}_{n_{r-1}}v^r_{n_{r}}v^{r-2}_{n_{r-2}} \cdots v^1_{n_1}u^0.
\end{align*}
Using the commutator formula (\ref{comm}) again on the right hand side of the equation above we get
\begin{align*}
v^{r-1}_{n_{r-1}}&v^r_{n_{r}}v^{r-2}_{n_{r-2}} \cdots v^1_{n_1}u^0=\\
&v^{r-1}_{n_{r-1}}\left(v^{r-2}_{n_{r-2}}v^r_{n_r}v^{r-3}_{n_{r-2}} \cdots v^1_{n_1}u^0+\sum_{i\geq 0}{n_r\choose i}(v^r_i v^{r-2})_{n_r+n_{r-2}-i}v^{r-3}_{n_{r-3}} \cdots v^1_{n_1}\right)u^0.
\end{align*} 

If there exists $i\geq0$ such that $v^{r-1}_{n_{r-1}}(v^r_i v^{r-2})_{n_r+n_{r-2}-i}v^{r-3}_{n_{r-3}} \cdots v^1_{n_1}u^0 \neq 0$ using the inductive hypothesis, we have that either
$$\sum_{j=1}^r (|v^j|-n_j-1)=0 \textrm{ or }  Re\left(\sum_{j=1}^r(|v^j|-n_j-1)\right)>0.$$ Moreover, this reasoning applies as long as there exists $1<j<r$ and $i\geq 0$ such that 
$$v^{r-1}_{n_{r-1}}v^{r-2}_{n_{r-2}}\cdots (v^r_iv^{r-j})_{n_{r}+n_{r-j}-i}\cdots v^1_{n_{1}}u^0\neq 0.$$
To finish the proof we show that there must exist such a $j$ and $i$: Otherwise, the commutator formula applied $r$ times implies that 
\begin{align*}
v^r_{n_r}v^{r-1}_{n_{r-1}}v^{r-2}_{n_{r-2}} \cdots v^1_{n_1}u^0=v^{r-1}_{n_{r-1}}v^{r-2}_{n_{r-2}} \cdots v^1_{n_1}v^r_{n_{r}}u^0\neq 0.
\end{align*}
In particular, $v^r_{n_{r}}u^0\neq 0$ which contradicts the fact that $u^0\in \Omega(V)$ since by assumption $|v^r|-n_r-1\neq 0$ and $Re(|v^r|-n_r-1)\leq 0$ in Case 2. Therefore, the lemma holds in case 2.II.
\end{enumerate}
\end{proof}

\end{document}